\documentclass[12pt]{amsart}
\usepackage{amsmath,amssymb,hyperref,geometry,graphicx,ytableau,fp,mathdots,amsthm,tikz,tikz-3dplot,enumitem}
\usetikzlibrary{cd}
\usepackage{pgfplots}
\usepackage{float}
\usepgfplotslibrary{fillbetween}
\usepackage[reftex]{theoremref}
\usetikzlibrary{arrows}
\usepackage{subfiles,bm}

\newcommand{\mylabel}[2]{#2\def\@currentlabel{#2}\label{#1}}

\newtheorem{theorem}[equation]{Theorem}
\newtheorem{lemma}[equation]{Lemma}

\newtheorem{proposition}[equation]{Proposition}
\theoremstyle{definition}
\newtheorem{example}[equation]{Example}

\newtheorem{definition}[equation]{Definition}
\newtheorem*{Result}{Theorem}

\numberwithin{equation}{section}

\newcommand{\R}{\mathbb{R}}

\newcommand{\Q}{\mathbb{Q}}

\newcommand{\cX}{\mathcal{X}}

\newcommand{\Vol}{\operatorname{Vol}}

\newcommand{\vol}{\mathrm{vol}}

\newcommand{\N}{\mathrm{N}} 

\newcommand{\Hess}{\mathrm{Hess}}

\newcommand{\bSigma}{{\bm{\Sigma}}}
\newcommand{\bX}{\bm{\mathcal{X}}}

\renewcommand{\phi}{\varphi}

\begin{document}
\title{Lorentzian fans}
\author[D.~Ross]{Dustin Ross}
\address{Department of Mathematics, San Francisco State University}
\email{rossd@sfsu.edu}

\begin{abstract}
We introduce the notion of Lorentzian fans, which form a special class of tropical fans that are particularly well-suited for proving Alexandrov--Fenchel type inequalities. To demonstrate the utility of Lorentzian fans, we prove a practical characterization of them in terms of their two-dimensional star fans. We also show that Lorentzian fans are closed under many common tropical fan operations, and we discuss how the Lorentzian property descends to the underlying tropical variety, allowing us to deduce Alexandrov--Fenchel type inequalities in the general setting of tropical intersection theory on tropical fan varieties.
\end{abstract}

\maketitle

\section{Introduction}

This paper investigates an analogue of the Alexandrov--Fenchel (AF) inequalities in the setting of tropical fans. More precisely, given a tropical $d$-fan $\bSigma=(\Sigma,\omega)$, comprised of a simplicial $d$-fan $\Sigma$ and a positive Minkowski $d$-weight $\omega$, there is an open cone $K(\Sigma)\subseteq D(\Sigma)$ of strictly convex divisors, and for $D_\heartsuit,D_\diamondsuit,D_3,\dots,D_d\in K(\Sigma)$, we aim to understand when their mixed degrees satisfy the following inequalities
\begin{equation}
\tag{AF}\deg_{\bSigma}(D_\heartsuit D_\diamondsuit D_3\cdots D_d)^2\geq \deg_\bSigma(D_\heartsuit^2D_3\cdots D_d)\cdot \deg_\bSigma(D_\diamondsuit^2D_3\cdots D_d).
\end{equation}
In the setting of complete fans, these inequalities are a special case of the classical Alexandrov--Fenchel inequalites for mixed volumes of polytopes, but recent developments have suggested the importance of generalizing these inequalities beyond the complete setting. For example, the celebrated resolution of the Heron--Rota--Welsh conjecture by Adiprasito, Huh, and Katz \cite{AdiprasitoHuhKatz} proceeds by way of proving AF inequalites in the setting where $\Sigma$ is the (incomplete, in general) Bergman fan of a matroid.

The ingenious approach to AF inequalities developed by Adiprasito, Huh, and Katz relied on building a notion of \emph{combinatorial Hodge theory} in the setting of Bergman fans, and the structure they developed has since been further refined and generalized in a number of ways \cite{BHMPW,BHMPW2,ArdilaDenhamHuh,AP1,AP2}. While the Hodge-theoretic approach has been powerful and illuminating,  it is also a much bigger hammer than is necessary if one is simply interested in AF inequalities, which follow from just one small part of the so-called K\"ahler package developed in \cite{AdiprasitoHuhKatz}. With an aim of developing more refined tools for studying AF inequalities, Br\"and\'en and Huh \cite{BrandenHuh} and Br\"and\'en and Leake \cite{BrandenLeake} have recently developed the notion of \emph{Lorentzian polynomials}, which hone in on the specific part of the K\"ahler package that is most relevant to AF inequalities. This paper builds on these recent developments of Lorentzian polynomials by introducing a notion of \emph{Lorentzian fans}, which we propose as a refined tool that gets at the heart of AF inequalities in the setting of tropical fans. As a tool for demonstrating AF inequalities in the setting of tropical fans, our thesis is that  the theory of Lorentzian fans introduced here is simple to employ and applies in settings that were previously inaccessible using combinatorial Hodge theory techniques.

To motivate the definition of Lorentzian fans, we make two important observations. First, we note that it has long been understood---indeed, it is already apparent in Aleksandrov's original proof \cite{Aleksandrov}---that a particularly useful way to prove AF inequalities is to show that, for all $D_3,\dots,D_d\in K(\Sigma)$, the quadratic form
\begin{align*}
\Phi:D(\Sigma)\times D(\Sigma)&\rightarrow\R\\
\nonumber (D_1,D_2)&\mapsto\deg_\bSigma(D_1 D_2 D_3\cdots D_d)
\end{align*}
has exactly one positive eigenvalue. To see why this implies the AF inequalities, consider the following $2\times 2$ principal minor of $\Phi$ associated to a fixed pair $D_\heartsuit,D_\diamondsuit\in K(\Sigma)$:
\[
M=\left[
\begin{array}{cc}
\deg_\bSigma(D_\heartsuit^2D_3\cdots D_d) & \deg_\bSigma(D_\diamondsuit D_\heartsuit D_3\cdots D_d)\\
\deg_\bSigma(D_\heartsuit D_\diamondsuit D_3\cdots D_d) & \deg_\bSigma(D_\diamondsuit^2D_3\cdots D_d)
\end{array}
\right].
\]
If $\Phi$ has exactly one positive eigenvalue, then Cauchy's interlacing theorem implies that $M$ has \emph{at most one positive eigenvalue}. On the other hand, elementary computations with the characteristic polynomial imply that $M$ has \emph{at least one positive eigenvalue}, simply because it is a symmetric $2\times 2$ matrix with positive entries. Thus, $M$ has \emph{exactly one positive eigenvalue}, implying that its determinant is nonpositive, and the AF inequalites follow.

The second observation that motivates us is that most proofs of the AF inequalities in the classical polytope setting---again, Aleksandrov's original proof is an example---use an induction argument, where the induction step reduces dimension by computing mixed volumes of polytopes in terms of the mixed volumes of their facets. Since faces of polytopes translate to stars of normal fans, this suggests that any study of AF inequalities in a more general tropical fan setting should be stable under taking star fans. Importantly, we note that if $\bSigma=(\Sigma,\omega)$ is a tropical fan and $\tau\in\Sigma$ is a cone, we can aways endow the star fan $\Sigma^\tau$ with a compatible tropical structure $\omega^\tau$ (only determined up to positive scaling).

From the two observations above, we now arrive at the definition of a \textbf{Lorentzian fan} (Definition~\ref{def:lorentzian}): it is a tropical fan $\bSigma$ with $K(\Sigma)\neq\emptyset$ such that the quadratic forms $\Phi$ associated to all of the stars $\bSigma^\tau$ have exactly one positive eigenvalue. The discussion above proves that mixed degrees of strictly convex divisors on Lorentzian fans satisfy AF inequalities (Proposition~\ref{prop:tropicalfanAF}), but what is not clear is whether the Lorentzian property is any easier to study than the AF inequalities themselves. Our main result is a practical characterization of Lorentzian fans, reducing the defining condition to the case of $2$-dimensional star fans.

\begin{Result}[Theorem~\ref{thm:characterization}]
A tropical fan $\bSigma$ with $K(\Sigma)\neq\emptyset$ is Lorentzian if and only if
\begin{enumerate}
\item[(A)] $\Sigma^\tau\setminus\{0\}$ is connected for every cone $\tau$ of codimension at least $2$, and
\item[(B)] every $2$-dimensional tropical star fan $\bSigma^\tau$ is Lorentzian.
\end{enumerate}
\end{Result}

We note that (A) is a mild connectedness assumption on $\Sigma$, disallowing, for example, fans that are locally a pair of $d$-dimensional cones meeting along a cone of dimension less than $d-1$, while (B) is completely concrete to check in practice. In particular, since (B) only concerns $2$-dimensional star fans, there is only one quadratic form $\Phi$ associated to $\bSigma^\tau$, independent of choosing any strictly convex divisors. Thus, verifying (B) amounts simply to computing the eigenvalues of a finite set of explicitly computable matrices. As a proof of concept, we note that (A) and (B) are relatively straightforward to verify for Bergman fans of matroids (Theorem~\ref{thm:Bergman}, see also \cite{BackmanEurSimpson,BrandenLeake,NowakOMelvenyRoss}).

In addition to the above characterization of Lorentzian fans, this paper develops a number of additional properties that make Lorentzian fans a particularly convenient class of tropical fans to work with; for example Lorentzian fans are closed under the operations of (i) passing to star fans, (ii) taking products, (iii) acting by strictly convex divisors, (iv) taking tropical modifications along strictly convex divisors, and (v) changing the fan structure on $\Sigma$. We also show that it is not hard to derive examples of Lorentzian fans that do not satisfy the suite of properties included in the K\"ahler package (Example~\ref{ex:lorentziannotlefschetz}); these are fans for which the theory of Lorentzian fans is applicable but the theory of combinatorial Hodge theory is not. Additionally, we note that Lorentzian fans are a strict subset of tropical fans: there exist examples of tropical fans that fail to be Lorentzian (Example~\ref{ex:tropicalnotlorentzian}, from \cite{BabaeeHuh}).

It is worth elaborating on point (v) in the previous paragraph and making a few observations regarding our conventions. A key principle of tropical geometry is that, given a tropical structure on a fan, it induces a canonical tropical structure on any fan with the same support, allowing one to define a \emph{tropical fan variety} as an equivalence class of tropical fans \cite{AllermanRau}. While tropical fans are often assumed to be rational, we do not make that assumption in this paper, and we instead work with simplicial, possibly irrational fans. However, as we discuss in Section~\ref{sec:variety}, the key observations regarding rational tropical varieties can be extended in a natural way to the simplicial, irrational setting, and this allows one to define tropical fan varieties in the irrational setting, as well. In this general setting, we define a \textbf{Lorentzian fan variety} to be a tropical fan variety for which at least one---and thus, by (v) above, all---of its representatives are Lorentzian fans. We prove (Theorem~\ref{thm:fansetAF}) that mixed degrees of divisors on Lorentzian fan varieties satisfy the AF inequalities, which demonstrates that Lorentzian fan varieties have a natural place within tropical intersection theory. In particular, these observations inspire further developments of the Lorentzian property for more general tropical varieties beyond those supported on fan sets.

\subsection*{Contents}

In Section~\ref{sec:setup}, we establish the background required to study mixed degrees of divisors on tropical fans, including Chow rings, Minkowski weights, and piecewise linear functions. There is some novelty in how we deal with the markings on the fans, given that our fans are not assumed to be rational. In Section~\ref{sec:lorentzianfans}, we give a precise definition of Lorentzian fans and relate them to the notion of Lorentzian polynomials on cones. Section~\ref{sec:characterization} is devoted to the statement and the proof of the two-dimensional characterization of Lorentzian fans. We then collect a number of results in Section~\ref{sec:properties} that describe common fan operations under which Lorentzian fans are closed. In Section~\ref{sec:variety}, we develop tropical fan varieties, following Allerman and Rau \cite{AllermanRau}, although our fans are not assumed to be rational, and we conclude the paper in Section~\ref{sec:lorentzianfansets} with the introduction and development of Lorentzian fan varieties.

\subsection*{Related work}

While in the final stages of preparing this paper, the author became aware that Petter Br\"and\'en and Jonathan Leake had independently been developing a project on hereditary Lorentzian polynomials on cones, and many of the results proved here are a consequence of their more general theory. We collectively agreed to post the two papers simultaneously, and it is this author's hope that this paper, written through the geometric lens of tropical fans and tropical intersection theory, will serve as a useful companion to the general polynomial theory developed in the work of Br\"and\'en and Leake.


\subsection*{Gratitude}

The author warmly acknowledges Anastasia Nathanson, Lauren Nowak, and Patrick O'Melveny for many insights that they shared while working on related collaborations \cite{NathansonRoss,NowakOMelvenyRoss}, and Federico Ardila, Emily Clader, Chris Eur, and Matt Larson for enlightening conversations that influenced this project. The author is grateful to Petter Br\"and\'en and Jonathan Leake for explaining their related project to him, and for the kindness they demonstrated while navigating the overlap between the two projects. This work is partially supported by a grant from the National Science Foundation: DMS--2001439.

\section{Algebraic structures on simplicial fans}\label{sec:setup}

In this section, we present an introduction to simplicial fans and natural algebraic structures associated to them, with an eye toward establishing notational conventions and the particular properties that will be relevant to our development of Lorentzian fans. Our primary aim is to introduce mixed degrees of divisors on tropical fans, interpreting them both algebraically using Chow rings and more geometrically using the action of piecewise linear functions on Minkowski weights. While we are unaware of a precise reference for these topics in the setting of simplicial, possibly irrational fans, the well-known properties from the rational setting carry through in a straightforward manner. Where appropriate, we provide references for justifications in the rational setting, but our goal in this introductory section is to present the material in such a way that all assertions can be viewed as practical exercises for a learner with some prior familiarity with algebra and polyhedral geometry.

\subsection{Fan conventions}

Let $V$ be an $n$-dimensional vector space over $\R$ with dual space $V^\vee$. For any $\phi\in V^\vee$, the associated \textbf{hyperplane} and \textbf{halfspace} in $V$ are defined by
\[
H_\phi=\{u\in V\mid \phi(u)= 0\}\;\;\;\text{ and }\;\;\;H_\phi^-=\{u\in V\mid \phi(u)\leq 0\},
\]
respectively. A \textbf{polyhedral cone} in $V$ is a finite intersection of halfspaces. 

Let $\sigma\subseteq V$ be a polyhedral cone. The \textbf{span} of $\sigma$, denoted $V_\sigma$, is the smallest subspace of $V$ containing $\sigma$. The \textbf{dimension} of $\sigma$, denoted $\dim(\sigma)$, is the vector space dimension of $V_\sigma$. The \textbf{relative interior} of $\sigma$, denoted $\sigma^\circ$, is the topological interior of $\sigma$ as a subset of $V_\sigma$. A \textbf{face} of $\sigma$ is any subset of the form
\[
\sigma\cap H_\phi\;\;\;\text{ where }\;\;\;\sigma\subseteq H_\phi^-.
\]
We use $\preceq$ to denote the face containment relation, and we denote the $k$-dimensional faces of $\sigma$ by $\sigma(k)$. We say that $\sigma$ is \textbf{strongly convex} if it does not contain a nonzero subspace of $V$, and we say that $\sigma$ is \textbf{simplicial} if it is strongly convex and $\dim(\sigma)=|\sigma(1)|$. A \textbf{fan in $V$} is a finite set $\Sigma$ of strongly convex polyhedral cones in $V$ such that
\begin{enumerate}
\item if $\sigma\in\Sigma$ and $\tau\preceq\sigma$, then $\tau\in\Sigma$, and
\item if $\sigma_1,\sigma_2\in\Sigma$, then $\sigma_1\cap\sigma_2\preceq \sigma_1$ and $\sigma_1\cap\sigma_2\preceq\sigma_2$.
\end{enumerate}

Let $\Sigma$ be a fan in $V$. The collection of $k$-dimensional cones of $\Sigma$ is denoted $\Sigma(k)$, and the \textbf{$k$-skeleton} of $\Sigma$, denoted $\Sigma[k]$, is the fan comprised of all cones in $\Sigma$ that have dimension at most $k$. We say that $\Sigma$ is \textbf{simplicial} if all of its cones are simplicial. We say that $\Sigma$ is \textbf{pure} if every inclusion-maximal cone of $\Sigma$ has the same dimension, and we say that $\Sigma$ is a \textbf{$d$-fan} if it is pure of dimension $d$. The \textbf{support} of $\Sigma$, denoted $|\Sigma|$, is the union of all $\sigma\in\Sigma$. We say that a fan $\Sigma_1$ is a \textbf{refinement} of $\Sigma_2$ if they have the same support and every cone of $\Sigma_1$ is a subset of some cone of $\Sigma_2$. A \textbf{marking} of $\Sigma$ is a choice of vector $u_\rho\in\rho^\circ$ for every $\rho\in\Sigma(1)$; we denote a choice of marking by $u=(u_\rho)_{\rho\in\Sigma(1)}$.

\subsection{Chow rings of simplicial fans}

Let $\Sigma$ be a simplicial $d$-fan in $V$ with marking $u$. The \textbf{Chow ring of $(\Sigma,u)$} is defined by
\[
A^\bullet(\Sigma,u)= \frac{\R\big[x_\rho\;|\;\rho\in\Sigma{(1)}\big]}{I_\Sigma+J_{\Sigma,u}}
\]
where
\[
I_\Sigma = \big\langle x_{\rho_1}\cdots x_{\rho_k}\mid\{\rho_1,\dots,\rho_k\}\not\subseteq\sigma\text{ for any }\sigma\in\Sigma\big\rangle
\]
and
\[
J_{\Sigma,u} = \bigg\langle \sum_{\rho\in\Sigma{(1)}}\phi(u_\rho) x_\rho\;\bigg|\;\phi\in V^\vee\bigg\rangle.
\]
As both $I_\Sigma$ and $J_{\Sigma,u}$ are homogeneous, the Chow ring $A^\bullet(\Sigma,u)$ is a graded ring, and we denote by $A^k(\Sigma,u)$ the subgroup of homogeneous elements of degree $k$. The algebra generators of $A^\bullet(\Sigma,u)$ are denoted $X_{u,\rho}= [x_{u,\rho}]\in A^1(\Sigma,u)$, and we extend this notation to any cone $\sigma\in\Sigma(k)$ by defining the associated \textbf{cone monomial}
\[
X_{u,\sigma} = \prod_{\rho\in\sigma(1)}X_{u,\rho}\in A^k(\Sigma,u).
\]
It follows from the simplicial hypothesis that $A^k(\Sigma,u)$ is spanned by cone monomials \cite[Proposition 5.5]{AdiprasitoHuhKatz}, and it then follows from the $d$-fan hypothesis that $A^k(\Sigma,u)=0$ for all $k>d$.

We say that two Chow classes $X\in A^\bullet(\Sigma,u)$ and $X'\in A^\bullet(\Sigma,u')$ are \textbf{equivalent} if there exists $\lambda\in\R_{>0}^{\Sigma(1)}$ such that 
\[
u'=\lambda u\;\;\;\text{ and }\;\;\;X'=\lambda^{-1}\cdot X,
\]
where the action on the right is defined on generators by $\lambda^{-1}\cdot X_{u,\rho}=\lambda_\rho^{-1}X_{u',\rho}$. A \textbf{Chow class on $\Sigma$} is an equivalence class of pairs $(X,u)$ where $X\in A^\bullet(\Sigma,u)$, and the \textbf{Chow ring of $\Sigma$}, denoted $A^\bullet(\Sigma)$, is the ring of Chow classes on $\Sigma$. For any choice of marking $u$, there is a canonical isomorphism $A^\bullet(\Sigma)\cong A^\bullet(\Sigma,u)$.

\subsection{Minkowski weights}\label{subsec:minkowskiweights}

Let $\Sigma$ be a simplicial $d$-fan in $V$ with marking $u$. A \textbf{Minkowski $k$-weight on $(\Sigma,u)$} is a function $\omega:\Sigma(k)\rightarrow\R$ such that, for every $\tau\in\Sigma(k-1)$, we have the following balancing condition
\begin{equation}\label{eq:balancingcondition}
\sum_{\sigma\in\Sigma(k)\atop \tau\prec\sigma}\omega(\sigma)u_{\sigma\setminus\tau}\in V_\tau,
\end{equation}
where $\sigma\setminus\tau\in\Sigma(1)$ represents the unique ray in $\sigma(1)\setminus\tau(1)$. We denote the vector space of Minkowski $k$-weights on $(\Sigma,u)$ by $MW_k(\Sigma,u)$ and the associated graded vector space by
\[
MW_\bullet(\Sigma,u)=\bigoplus_{k=0}^d MW_k(\Sigma,u).
\]
The balancing condition on a Minkowski $k$-weight $\omega\in MW_k(\Sigma,u)$ is sufficient and necessary for there to exist a linear map $f_{\Sigma,u,\omega}:A^k(\Sigma,u)\rightarrow\R$ such that $f_{\Sigma,u,\omega}(X_{u,\sigma})=\omega(\sigma)$ for every $\sigma\in\Sigma(k)$ \cite[Proposition~5.6]{AdiprasitoHuhKatz}, and this gives an isomorphism of graded vector spaces
\begin{align*}
MW_\bullet(\Sigma,u)&\stackrel{\cong}{\longrightarrow} A^\bullet(\Sigma,u)^\vee\\
\omega&\longmapsto f_{\Sigma,u,\omega}.
\end{align*}

We say that two Minkowski $k$-weights $\omega\in MW_k(\Sigma,u)$ and $\omega'\in MW_k(\Sigma,u')$ are \textbf{equivalent} if there exists $\lambda\in\R_{>0}^{\Sigma(1)}$ such that 
\[
u'=\lambda u\;\;\;\text{ and }\;\;\;\omega'=\lambda^{-1}\cdot \omega
\]
where, for every $\tau\in\Sigma(k)$, we define
\[
(\lambda^{-1}\cdot \omega)(\tau)=\big(\prod_{\rho\in\tau(1)}\lambda_\rho^{-1}\big)\omega(\tau).
\]
A \textbf{Minkowski $k$-weight on $\Sigma$} is an equivalence class of pairs $(\omega,u)$ where $\omega\in MW_\bullet(\Sigma,u)$, and the \textbf{vector space of Minkowski weights on $\Sigma$}, denoted $MW_\bullet(\Sigma)$, comprises all Minkowski weights on $\Sigma$. For any $u$, there is a canonical isomorphism $MW_\bullet(\Sigma)\cong MW_\bullet(\Sigma,u)$, and the vector space isomorphism $MW_\bullet(\Sigma,u)\cong A^\bullet(\Sigma,u)$ for any marking $u$ induces a canonical isomorphism 
\begin{align*}
MW_\bullet(\Sigma)&\stackrel{\cong}{\longrightarrow} A^\bullet(\Sigma)^\vee\\
\omega&\longmapsto f_{\Sigma,\omega}.
\end{align*}

\subsection{Tropical divisors}

Let $\Sigma$ be a simplicial $d$-fan in $V$. A \textbf{piecewise linear function on $\Sigma$} is a continuous function $\phi:|\Sigma|\rightarrow \R$ such that, for every $\sigma\in\Sigma$, there is a linear map $\phi_\sigma\in V^\vee$ such that $\phi|_\sigma=\phi_\sigma$. Let $PL(\Sigma)$ denote the vector space of piecewise linear functions on $\Sigma$, and let $L(\Sigma)\subseteq PL(\Sigma)$ denote the subspace of linear functions---that is, restrictions to $|\Sigma|$ of linear functions $\phi\in V^\vee$. The \textbf{vector space of divisors on $\Sigma$} is the quotient
\[
D(\Sigma)=\frac{PL(\Sigma)}{L(\Sigma)}.
\]

Upon choosing a marking $u$ of $\Sigma$, a piecewise linear function on $\Sigma$ is determined uniquely by its values at the marks. Let $D_{u,\rho}\in D(\Sigma)$ denote the class of the piecewise linear function that takes value $1$ at $u_\rho$ and value $0$ at $u_\eta$ for $\eta\neq\rho$. Then $\{D_{u,\rho}\mid\rho\in\Sigma(1)\}$ spans $D(\Sigma)$, and there is a vector space isomorphism
\begin{align*}
D(\Sigma)&\stackrel{\cong}{\longrightarrow}A^1(\Sigma,u)\\
D_{u,\rho}&\longmapsto X_{u,\rho}.
\end{align*}
These isomorphisms are compatible with scaling the markings, so they descend to a canonical isomorphism $D(\Sigma)\cong A^1(\Sigma)$; given a divisor $D\in D(\Sigma)$, let $X_D\in A^1(\Sigma)$ denote the corresponding Chow class.

The vector space $A^\bullet(\Sigma,u)^\vee$ can naturally be viewed as an $A^1(\Sigma,u)$-module via the action
\begin{align*}
A^1(\Sigma,u)\times A^\bullet(\Sigma,u)^\vee&\rightarrow A^{\bullet-1}(\Sigma,u)^\vee\\
(X,f)&\mapsto f(X\cdot-)
\end{align*}
where $ f(X\cdot-)$ simply precomposes the map $f$ with multiplication by $X$. Viewed instead as an action of $D(\Sigma)$ on $MW_\bullet(\Sigma,u)$, one carefully traces back through the vector space isomorphisms to see that the action of $A^1(\Sigma,u)$ on $A^\bullet(\Sigma,u)^\vee$ induces the action
\begin{align*}
D(\Sigma)\times MW_\bullet(\Sigma,u)&\rightarrow MW_{\bullet-1}(\Sigma,u)\\
(D,\omega)&\mapsto D\cdot\omega,
\end{align*}
where, for any $\omega\in MW_k(\Sigma,u)$ and $\tau\in\Sigma(k-1)$, we define
\[
(D\cdot\omega)(\tau)=\sum_{\sigma\in\Sigma(k)\atop\sigma\succ\tau}\phi_\sigma\big(\omega(\sigma) u_{\sigma\setminus\tau}\big)-\phi_\tau\Big(\sum_{\sigma\in\Sigma(k)\atop\sigma\succ\tau}\omega(\sigma) u_{\sigma\setminus\tau}\Big),
\]
where $\phi$ is any piecewise linear function representing $D$ \cite{AllermanRau}. Intuitively, one can think of $(D\cdot\omega)(\tau)$ as a measure of the failure of $\phi$ to be linear at $\tau$.

The actions of $A^1(\Sigma,u)$ on $A^\bullet(\Sigma,u)^\vee$ and $D(\Sigma)$ on $MW_\bullet(\Sigma,u)$ are compatible with equivalence of Chow classes and Minkowski weights, so they descend to give canonical actions
\[
A^1(\Sigma)\times A^\bullet(\Sigma)^\vee\rightarrow A^{\bullet-1}(\Sigma)^\vee\;\;\;\text{ and }\;\;\;D(\Sigma)\times MW_\bullet(\Sigma)\rightarrow MW_{\bullet-1}(\Sigma)
\]
that are compatible with the isomorphisms $A^1(\Sigma)\cong D(\Sigma)$ and $A^\bullet(\Sigma)^\vee\cong MW_\bullet(\Sigma)$.

\subsection{Mixed degrees of divisors}\label{subsec:mixeddegrees}

Let $\Sigma$ be a simplicial $d$-fan and let $\omega\in MW_d(\Sigma)$ be a Minkowski $d$-weight. Given divisors $D_1,\dots,D_d\in D(\Sigma)$, notice that $D_1\cdots D_d\cdot\omega$ is a Minkowski $0$-weight, which is nothing more than a function from the origin to $\R$. We define the \textbf{mixed degree of $D_1,\dots,D_d\in D(\Sigma)$ with respect to $(\Sigma,\omega)$} by
\begin{equation}\label{eq:degminkowski}
\deg_{\Sigma,\omega}(D_1\cdots D_d)=(D_1\cdots D_d\cdot\omega)(0)\in\R.
\end{equation}
Equivalently, using the compatibility of the actions in the previous subsection, one may define mixed degrees of divisors Chow-theoretically by
\begin{equation}\label{eq:degchow}
\deg_{\Sigma,\omega}(D_1\cdots D_d)=f_{\Sigma,\omega}(X_{D_1}\cdots X_{D_d})\in\R.
\end{equation}
When working with a specific representative $\omega\in MW_d(\Sigma,u)$ associated to a marking $u$, we will often include $u$ in the notation and write $\deg_{\Sigma,u,\omega}$.

While studying mixed degrees of divisors, it is useful to keep in mind both of the perspectives \eqref{eq:degminkowski} and \eqref{eq:degchow}: the perspective in terms of Minkowski weights is more geometric, related to how piecewise linear functions bend across the faces of $\Sigma$, while the perspective in terms of the Chow ring is more algebraic, computable by using the relations in $A^\bullet(\Sigma)$ to manipulate polynomials in the generators. We also note that the Chow ring perspective is motivated by intersection theory on toric varieties, so this perspective benefits from a wealth of intuition from algebraic geometry, intuition that has certainly influenced these developments but that is not a prerequisite to understanding them.

\subsection{Star fans}

Star fans are a useful tool for computing mixed degrees of divisors recursively in dimension, as we now describe. Suppose that $\Sigma$ is a simplicial $d$-fan in $V$. Given a cone $\tau\in\Sigma$, define the \textbf{neighborhood of $\tau$ in $\Sigma$} by
\[
N_\tau\Sigma = \{\pi \mid \pi\preceq \sigma \text{ for some }\sigma\in\Sigma \text{ with }\tau\preceq\sigma\}.
\]
The \textbf{star of $\Sigma$ at $\tau\in\Sigma$} is the fan in $V^\tau=V/V_\tau$ comprised of all quotients of cones in the neighborhood of $\tau$:
\[
\Sigma^\tau = \{\overline\pi\mid \pi\in N_\tau\Sigma\},
\]
where $\overline\pi\subseteq V^\tau$ denotes the quotient of the cone $\pi\subseteq V$ by $V_\tau$. The assumption that $\Sigma$ is a simplicial $d$-fan implies that $\Sigma^\tau$ is a simplicial $d^\tau$-fan, where $d^\tau=d-\dim(\tau)$. 

Given $\omega\in MW_d(\Sigma,u)$, we obtain a Minkowski weight $\omega^\tau\in MW_{d^\tau}(\Sigma^\tau,u^\tau)$ as follows: 
\begin{itemize}
\item every ray $\eta\in\Sigma^\tau(1)$ is the quotient of a unique ray $\hat\eta\in N_\tau\Sigma(1)$; define $u^\tau_\eta=\overline u_{\hat\eta}$;
\item every maximal cone $\sigma\in\Sigma^\tau(d^\tau)$ is the quotient of a unique maximal cone $\hat\sigma\in N_\tau\Sigma(d)$; define $\omega^\tau(\sigma)=\omega(\hat\sigma)$.
\end{itemize}
Importantly, we note that equivalent Minkowski weights on $\Sigma$ \emph{do not} generally induce equivalent Minkowski weights on $\Sigma^\tau$, so the choice of marking is essential in our discussion of star fans. In particular, if $u_2=\lambda u_1$ are two markings and $\omega_1\in MW_d(\Sigma,u_1)$ and $\omega_2\in MW_d(\Sigma,u_2)$ are equivalent Minkowski weights in $MW_d(\Sigma)$, then one readily checks that
\begin{equation}\label{eq:starscale}
\omega_2^\tau=\big(\prod_{\rho\in\tau(1)}\lambda_\rho^{-1}\big)(\lambda^{-1}\cdot\omega_1^\tau).
\end{equation}
If the product in the right-hand side of \eqref{eq:starscale} is not one, then $\omega_1^\tau$ is not equivalent to $\omega_2^\tau$. Thus, a Minkowski weight $\omega\in MW_d(\Sigma)$ \emph{does not} determine a unique Minkowski weight in $MW_d(\Sigma^\tau)$; rather, it determines a family of Minkowski weights that are related by positive scaling. Since positive scalings of Minkowski weights do not affect the properties we are most interested in---such as Alexandrov--Fenchel type inequalities---this will not be a serious issue, but it is certainly worth noting.

For every cone $\sigma\in N_\tau\Sigma$, let $\sigma\setminus\tau$ denote the cone with rays $\sigma(1)\setminus\tau(1)$. Define the \textbf{boundary of $\tau$} to be the subfan
\[
B_\tau\Sigma=\{\sigma\setminus\tau\mid\sigma\in N_\tau\Sigma\},
\]
and note that the quotient map $V\rightarrow V^\tau$ induces a cone-wise bijection of fans $B_\tau\Sigma\rightarrow \Sigma^\tau$. Given any Minkowski $k$-weight $\gamma\in MW_k(N_\tau\Sigma,u)$ with $k\leq d^\tau$, we obtain a Minkowski $k$-weight $\overline{\gamma}\in MW_k(\Sigma^\tau,u^\tau)$ defined by
\[
\overline{\gamma}(\overline\pi)=\gamma(\pi),
\]
where $\pi\in B_\tau\Sigma$ and $\overline\pi$ is its image in $\Sigma^\tau$. This induces, for any $k\leq d^\tau$,  a linear map $MW_k(N_\tau\Sigma,u)\rightarrow MW_k(\Sigma^\tau,u^\tau)$ that sends $\gamma$ to $\overline\gamma$. Given $D\in D(\Sigma)$, we can always choose a piecewise linear representative $\phi$ such that $\phi_\tau=0$, and any such $\phi$ descends to a piecewise linear function $\overline\phi$ on $\Sigma^\tau$; let $\overline D$ denote the divisor on $\Sigma^\tau$ represented by $\overline\phi$. This gives us a linear map $D(\Sigma)\rightarrow D(\Sigma^\tau)$ that sends $D$ to $\overline D$. It follows from the definition of the $D(\Sigma)$-action that the linear maps $MW_k(N_\tau\Sigma,u)\rightarrow MW_k(\Sigma^\tau,u^\tau)$ and $D(\Sigma)\rightarrow D(\Sigma^\tau)$ are compatible with the associated module structures in the following sense:
\begin{equation}\label{eq:compatiblemodulestructurestar}
\overline{D\cdot\gamma}=\overline{D}\cdot\overline\gamma\in MW_{k-1}(\Sigma^\tau,u^\tau).
\end{equation}

Label the rays of $\tau$ by $\tau(1)=\{\rho_1,\dots,\rho_k\}$. Given any Minkowski $d$-weight $\omega\in MW_d(\Sigma,u)$, it is straightforward to see that $D_{u,\rho}\cdot\omega$ is supported on $N_\rho\Sigma(d-1)$ for any ray $\rho\in\Sigma(1)$, and iterating this process, it follows that $D_{u,\rho_1}\cdots D_{u,\rho_k}\cdot\omega$ is supported on $N_\tau\Sigma(d^\tau)$. Furthermore, by analyzing the $D(\Sigma)$-action more closely, we see that, for any $\pi\in B_\tau\Sigma(d_\tau)$, we have
\[
(D_{u,\rho_1}\cdots D_{u,\rho_k}\cdot\omega)(\pi)=\omega^\tau(\pi).
\]
It then follows from \eqref{eq:compatiblemodulestructurestar} that, for any $D_{1},\dots,D_{\ell}\in D(\Sigma)$, we have
\[
\overline{D_{1}\cdots D_{\ell}\cdot D_{u,\rho_1}\cdots D_{u,\rho_k}\cdot\omega}=\overline D_{1}\cdots\overline D_\ell\cdot \omega^\tau\in MW_{d-k-\ell}(\Sigma^\tau,u^\tau). 
\]
In particular, if $\ell=d-k$, we have argued that
\begin{equation}\label{eq:reduction}
\deg_{\Sigma,u,\omega}(D_1\cdots D_{d-k}\cdot D_{u,\rho_1}\cdots D_{u,\rho_{k}})=\deg_{\Sigma^\tau,u^\tau,\omega^\tau}(\overline D_1\cdots \overline D_{d-k}).
\end{equation}
The upshot of \eqref{eq:reduction} is that, upon choosing a marking $u$, it allows us to reduce mixed degree computations from $\Sigma$ to $\Sigma^\tau$, providing a means for inductive arguments on dimension.

For the convex geometer, we note that, in the complete fan setting, \eqref{eq:reduction} is equivalent to a standard result that reduces the mixed volumes of a collection of strongly isomorphic polytopes to the mixed volumes on the facets of those polytopes \cite[Lemma~5.1.5]{Schneider}. For the algebraic geometer, we note that, in the rational complete fan setting, \eqref{eq:reduction} is a consequence of the projection formula applied to the inclusion of the torus-invariant subvariety corresponding to the star fan.

\subsection{Convex divisors}\label{subsec:convex}

Let $\Sigma$ be a simplicial $d$-fan. We say that a divisor $D\in D(\Sigma)$ is \textbf{convex} if, for every $\tau\in\Sigma$, there exists a representative $\phi$ that vanishes on $\tau$ and is nonnegative on its neighborhood $N_\tau\Sigma$. We say that $D$ is \textbf{strictly convex} if, for every $\tau\in\Sigma$, there exists a representative $\phi$ that vanishes on $\tau$ and is strictly positive on $N_\tau\Sigma\setminus\tau$. Convex divisors form a closed convex cone $\overline K(\Sigma)\subseteq D(\Sigma)$ whose interior is the set of strictly convex divisors $K(\Sigma)=\overline K(\Sigma)^\circ$ \cite[Proposition~4.3]{AdiprasitoHuhKatz}. Borrowing terminology from toric geometry, we say that $\Sigma$ is \textbf{quasiprojective} if $K(\Sigma)\neq\emptyset$. 

We say that a Minkowski weight is \textbf{nonnegative/positive} if all of its values are nonnegative/positive. It follows from the definition of the $D(\Sigma)$-action that
\begin{itemize}
\item $D\cdot\omega$ is nonegative whenever $\omega$ is nonnegative and $D\in \overline K(\Sigma)$, and
\item $D\cdot\omega$ is positive whenever $\omega$ is positive and $D\in K(\Sigma)$.
\end{itemize}
Intuitively, this just means that convex divisors are represented by piecewise linear functions whose graphs only bend upward, while strictly convex functions are additionally required to be bend upward along every cone of $\Sigma$. It follows from the above bullet points that
\begin{equation}\label{eq:nonnegativity}
\deg_{\Sigma,\omega}(D_1\cdots D_d)\geq 0\;\;\;\text{ if $\omega\in MW_d(\Sigma)$ is nonnegative and $D_1,\dots,D_d\in\overline K(\Sigma)$}
\end{equation}
and
\begin{equation}\label{eq:positivity}
\deg_{\Sigma,\omega}(D_1\cdots D_d)> 0\;\;\;\text{ if $\omega\in MW_d(\Sigma)$ is positive and $D_1,\dots,D_d\in K(\Sigma)$.}
\end{equation}

Positive Minkowski weights in $MW_d(\Sigma)$ are particularly relevant to our story. We define a \textbf{tropical $d$-fan $\bSigma$} to be a pair $\bSigma=(\Sigma,\omega)$ where $\Sigma$ is a simplicial $d$-fan and $\omega\in MW_d(\Sigma)$ is a positive Minkowski $d$-weight. Given a marking $u$ of $\Sigma$, we write $\bSigma(u)=(\Sigma,u,\omega)$ for the canonical representative on the marked fan $(\Sigma,u)$. Given a cone $\tau\in\Sigma$, we denote the marked tropical structure induced by $\bSigma(u)$ on the star fan $\Sigma^\tau$ by $\bSigma(u)^\tau=(\Sigma^\tau,u^\tau,\omega^\tau).$

\section{Lorentzian fans}\label{sec:lorentzianfans}

We now introduce the central definition of this paper.

\begin{definition}\label{def:lorentzian}
A marked tropical $d$-fan $\bSigma(u)=(\Sigma,u,\omega)$ is said to be \textbf{Lorentzian} if $\Sigma$ is quasiprojective and, for each $\tau\in\Sigma$ and all $D_3,\dots,D_{d^\tau}\in K(\Sigma^\tau)$, the quadratic form
\begin{align*}
D(\Sigma^\tau)\times D(\Sigma^\tau)&\rightarrow\R\\
(D_1,D_2)&\mapsto \deg_{\bSigma(u)^\tau}(D_1D_2D_3\cdots D_{d^\tau})
\end{align*}
has signature $(1,q,r)$.
\end{definition}

We recall that the \textbf{signature} $(p,q,r)$ of a quadratic form records the number---counted with multiplicity---of positive, negative, and zero eigenvalues, respectively, of any symmetric matrix representing the quadratic form with respect to some basis. That the signature does not depend on the choice of basis is a consequence of Sylvester's law of inertia. 

By \eqref{eq:starscale}, we see that different choices of $u$ result in quadratic forms that differ by a positive scalar, and since positive scaling does not affect the signature of a quadratic form, being Lorentzian does not depend on $u$. Consequently, we say that a tropical fan $\bSigma$ is \textbf{Lorentzian} if $\bSigma(u)$ is Lorentzian for some---and, thus, for every---marking $u$. 

Note that every star of a star is, itself, a star; more precisely, $(\Sigma^\tau)^\pi=\Sigma^{\hat\pi}$, where $\hat\pi\in\Sigma$ is the unique cone containing $\tau$ and whose quotient is $\pi$. It then follows that a tropical fan is Lorentzian if and only if all of its stars are Lorentzian. Moreover, we observe that the stipulation on the stars in Definition~\ref{def:lorentzian} is vacuously true for star fans of dimension zero and one, and it follows that a tropical fan is Lorentzian if and only all of its stars of dimension at least two are Lorentzian.

One of the primary motivations for studying Lorentzian fans is the following observation. 

\begin{proposition}\label{prop:tropicalfanAF}
Let $\bSigma$ be a Lorentzian $d$-fan. For any $D_1,\dots,D_d\in \overline K(\Sigma)$, we have
\[
\tag{AF}\deg_{\bSigma}(D_1 D_2 D_3\cdots D_d)^2\geq \deg_{\bSigma}(D_1^2D_3\cdots D_d)\cdot \deg_{\bSigma}(D_2^2D_3\cdots D_d).
\]
Furthermore, for any $D_1,D_2\in \overline K(\Sigma)$, the sequence 
\[
\big(\deg_{\bSigma}(D_1^kD_2^{d-k}\big)_{k=0}^d
\]
is unimodal and log-concave.
\end{proposition}

We recall that a nonnegative sequence $(a_0,\dots,a_d)$ is \textbf{unimodal} if 
\[
a_0\leq \dots \leq a_\ell \geq \dots \geq a_d \;\;\;\text{ for some }\;\;\;\ell\in\{0,\dots,d\}
\]
and \textbf{log-concave} if
\[
a_k^2\geq a_{k-1}a_{k+1}\;\;\;\text{ for all }\;\;\;k\in\{1,\dots,d-1\}.
\]

\begin{proof}[Proof of Proposition~\ref{prop:tropicalfanAF}]
If we assume that $D_1,\dots,D_d\in K(\Sigma)$, then (AF) was justified in the introduction, the log-concavity assertion is just a special case of (AF), and the unimodality assertion follows from the observation that every log-concave sequence of positive numbers is unimodal. To extend these statements to $D_i\in\overline K(\Sigma)$, we approximate each $D_i$ by an element of $K(\Sigma)=\overline K(\Sigma)^\circ$ (using the fact that $K(\Sigma)\neq\emptyset$), and then take a limit, noting that all of the desired assertions are closed conditions and therefore pass to the limit.
\end{proof}

Definition~\ref{def:lorentzian} is inspired, in part, by recent developments of Br\"and\'en and Leake regarding Lorentzian polynomials on cones \cite{BrandenLeake}. We now recall the main definition from \cite{BrandenLeake} in order to make this connection precise.

Suppose that $C\subseteq\R^n$ is a nonempty open convex cone, and let $f\in\R[x_1,\dots,x_n]$ be a homogeneous polynomial of degree $d$. For each $i\in\{1,\dots,n\}$ and $v=(v_1,\dots,v_n)\in\R^n$, we use the following shorthand for partial  and directional derivatives:
\[
\partial_i=\frac{\partial}{\partial x_i}\;\;\;\text{ and }\;\;\;\partial_v=\sum_{i=1}^n v_i\partial_i.
\]
Following Br\"and\'en and Leake, we say that $f$ is \textbf{$C$-Lorentzian} if, for all $v_1,\dots,v_d\in C$, 
\begin{enumerate}
\item[(P)] $\partial_{v_1}\cdots\partial_{v_d}f>0$, and
\item[(H)] the quadratic form
\begin{align*}
\R^n\times\R^n&\rightarrow\R\\
(x,y)&\mapsto\partial_x\partial_y\partial_{v_3}\cdots\partial_{v_d}f
\end{align*}
has exactly one positive eigenvalue.
\end{enumerate}
More concretely, (H) is equivalent to the condition that the Hessian of the quadratic form
\[
\partial_{v_3}\cdots\partial_{v_d}f\in\R[x_1,\dots,x_n]
\]
has exactly one positive eigenvalue, which is also equivalent to the condition that there exists an invertible linear change of variables $\ell_1(x),\dots,\ell_n(x)$ such that
\[
\partial_{v_3}\cdots\partial_{v_d}f=\ell_1(x)^2-\ell_2(x)^2-\cdots-\ell_k(x)^2
\]
for some $k\in\{2,\dots,n\}$.

To connect the definition of Lorentzian fans in Definition~\ref{def:lorentzian} with the definition of $C$-Lorentzian functions, we require a few additional notions. Given a simplicial $d$-fan $\Sigma$, consider the vector space $\R^{\Sigma(1)}$ with basis vectors $\{e_\rho\mid\rho\in\Sigma(1)\}$. A general element of $\R^{\Sigma(1)}$ is written
\[
z=\sum_{\rho\in\Sigma(1)}z_\rho e_\rho.
\]
Given a marking $u$ on $\Sigma$, there is a natural exact sequence
\[
V^\vee\rightarrow \R^{\Sigma(1)}\rightarrow D(\Sigma)\rightarrow 0
\]
where the maps are given by
\[
\phi\mapsto z_u(\phi)=\sum_{\rho\in\Sigma(1)}\phi(u_\rho) e_\rho\;\;\;\text{ and }\;\;\; z\mapsto D_u(z)=\sum_{\rho\in\Sigma(1)}z_\rho D_{u,\rho}
\]
Given a Minkowski weight $\omega\in MW_d(\Sigma,u)$, define the \textbf{volume polynomial of $(\Sigma,u,\omega)$} by
\begin{align*}
\Vol_{\Sigma,u,\omega}:\R^{\Sigma(1)}&\rightarrow \R\\
z&\mapsto \deg_{\Sigma,u,\omega}(D_u(z)^d).
\end{align*}
Alternatively, using the equivalent definitions \eqref{eq:degminkowski} and \eqref{eq:degchow} of mixed degrees, we may also realize the volume polynomial in terms of Chow classes:
\[
\Vol_{\Sigma,u,\omega}(z)=f_{\Sigma,u,\omega}(X_u(z)^d)\;\;\;\text{ where }\;\;\;X_u(z)=\sum_{\rho\in\Sigma(1)}z_\rho X_{u,\rho}\in A^1(\Sigma).
\]
Note that $\Vol_{\Sigma,u,\omega}$ is a homogeneous polynomial of degree $d$ in $\R[z_\rho\mid\rho\in\Sigma(1)]$ and it vanishes on the image of $V^\vee$. Consider the open cone $K_u(\Sigma)\subseteq\R^{\Sigma(1)}$ consisting of all $z\in\R^{\Sigma(1)}$ such that $D_u(z)\in K(\Sigma)$. The characterization in the next result builds a concrete bridge between Definition~\ref{def:lorentzian} and the concept of $C$-Lorentzian functions.

\begin{proposition}\label{prop:lorentzians}
If $\bSigma(u)$ is a quasiprojective tropical $d$-fan, then $\bSigma(u)$ is Lorentzian if and only if, for every $\tau\in\Sigma$, the volume polynomial $\Vol_{\bSigma(u)^\tau}$ is $K_{u^\tau}(\Sigma^\tau)$-Lorentzian.
\end{proposition}

\begin{proof}
For any $z_1,\dots,z_{k}\in\R^{\Sigma(1)}$, notice that
\begin{equation}\label{eq:partials}
\partial_{z_1}\cdots\partial_{z_k}\Vol_{\bSigma(u)}=\frac{d!}{(d-k)!}\deg_{\bSigma(u)}(D_u(z)^{d-k}\cdot D_u(z_1)\cdots D_u(z_{k})). 
\end{equation}
Thus, from \eqref{eq:positivity}, we see that
\[
\partial_{z_1}\cdots\partial_{z_d}\Vol_{\bSigma(u)}>0\;\;\;\text{ for all }\;\;\;z_1,\dots,z_d\in K_u(\Sigma),
\] 
showing that $\Vol_{\bSigma(u)}$ satisfies (P), without assuming that $\bSigma(u)$ is Lorentzian. Since $\Vol_{\bSigma(u)}$ vanishes on the image of $V^\vee$, it follows from \eqref{eq:partials} that the quadratic form
\begin{align*}
\R^{\Sigma(1)}\times \R^{\Sigma(1)}&\rightarrow\R\\
(x,y)&\mapsto\partial_x\partial_y\partial_{z_3}\cdots\partial_{z_d}\Vol_{\bSigma(u)},
\end{align*}
descends to $\frac{1}{2}d!$ times the quadratic form
\begin{align*}
D(\Sigma)\times D(\Sigma)&\rightarrow\R\\
(D_1,D_2)&\mapsto \deg_{\bSigma(u)}(D_1D_2D_u(z_3)\cdots D_u(z_d)),
\end{align*}
so the two quadratic forms have the same signature. Thus, $\Vol_{\bSigma(u)}$ satisfies (H) if and only if the latter quadratic form has exactly one positive eigenvalue, and the proposition then follows from applying these arguments to $\bSigma(u)^\tau$ for every $\tau\in\Sigma$.
\end{proof}

Because the assertion that $\Vol_{\bSigma(u)}$ is $K_u(\Sigma)$-Lorentzian is independent of the choice of $u$, we often abbreviate it and simply write that $\Vol_\bSigma$ is $K(\Sigma)$-Lorentzian.

We close this section with two illustrative examples. The first is an example of a Lorentzian fan that is not Lefschetz, in the sense of Ardila, Denham, and Huh \cite{ArdilaDenhamHuh}.

\begin{example}[A Lorentzian fan that is not Lefschetz]\label{ex:lorentziannotlefschetz}
Let $\Sigma$ be the $2$-skeleton of the coordinate subspaces in $\R^3$, marked by the vectors $\pm\mathrm{e}_1$, $\pm\mathrm{e}_2$, and $\pm\mathrm{e}_3$, and let $\omega:\Sigma(2)\rightarrow\R$ be the constant function that maps each cone to $1$. One readily verifies that $\omega\in MW_2(\Sigma,u)$, so we obtain a marked tropical fan $\bSigma(u)=(\Sigma,u,\omega)$.

Denote the 12 top-dimensional cones of $\Sigma$ by $\sigma_{1,2}^i$, $\sigma_{1,3}^i$, $\sigma_{2,3}^i$
where $i\in\{1,2,3,4\}$ denotes the respective quadrant in the coordinate plane determined by the indices in the subscripts. For any values $a,b,c\in\R$, one checks that the function $\omega_{a,b,c}:\Sigma(2)\rightarrow\R$ defined by
\[
\omega_{a,b,c}(\sigma_{1,2}^i)=ab,\;\;\;\omega_{a,b,c}(\sigma_{1,3}^i)=ac,\;\;\;\omega_{a,b,c}(\sigma_{2,3}^i)=bc
\]
is a Minkowski $2$-weight on $\Sigma$, and it follows that $\dim_\R(A^2(\Sigma))=\dim_\R(MW_2(\Sigma)) \geq 3$. In particular, since the top-degree Chow group is not one-dimensional, $A^\bullet(\Sigma)$ is not a Poincar\'e duality algebra, showing that $\bSigma(u)$ cannot be Lefschetz in the sense of \cite{ArdilaDenhamHuh}. On the other hand, denoting the rays of $\Sigma$ by $\Sigma(1)=\{\rho_1^\pm,\rho_2^\pm,\rho_3^\pm\}$, the volume polynomial can be computed explicitly:
\begin{align*}
\Vol_{\bSigma(u)}&=2(z_{\rho_1^+}+z_{\rho_1^-})(z_{\rho_2^+}+z_{\rho_2^-})+2(z_{\rho_1^+}+z_{\rho_1^-})(z_{\rho_3^+}+z_{\rho_3^-})+2(z_{\rho_2^+}+z_{\rho_2^-})(z_{\rho_3^+}+z_{\rho_3^-})\\
&=\frac{1}{2}(z_1+z_2+2z_3)^2-\frac{1}{2}(z_1-z_2)^2-2(z_3)^2,
\end{align*}
where $z_1=(z_{\rho_1^+}+z_{\rho_1^-})$ and similarly for $z_2$ and $z_3$. It follows that $\bSigma(u)$ is Lorentzian.

\end{example}

The next example, due to Babaee and Huh \cite{BabaeeHuh}, shows that not all tropical fans are Lorentzian.

\begin{example}[A tropical fan that is not Lorentzian]\label{ex:tropicalnotlorentzian}
In \cite[Section~5]{BabaeeHuh}, Babaee and Huh construct a two-dimensional fan $\Sigma$ in $\R^4$ by performing certain modifications to the fan over a realization of the complete bipartite graph $K_{4,4}$. Their construction is rather intricate, so we do not describe the details, but we note that it is a rational fan and marked by the primitive ray generators, and it is tropical with respect to a positive Minkowski weight $\omega\in MW_2(\Sigma,u)$. Their main result concerning this fan (Theorem 5.1) implies that $\Hess(\Vol_{\bSigma(u)})$ has more than one positive eigenvalue. Thus, $\bSigma(u)$ is tropical but not Lorentzian.

\end{example}

\section{A two-dimensional characterization of Lorentzian fans}\label{sec:characterization}

We now describe a more practical characterization of Lorentzian tropical fans, which essentially reduces the verification of the Lorentzian property to checking it for just the two-dimensional stars. Before stating this result, we first introduce a property of $d$-fans that is essential to our characterization. 

Given a $d$-fan $\Sigma$ in $V$ and a cone $\tau\in\Sigma$ with $\dim(\tau)\leq d-2$, we say that $\Sigma$ is \textbf{pinched at $\tau$} if $|\Sigma^\tau|\setminus\{0\}$ is disconnected (or equivalently, if $|N_\tau\Sigma|\setminus\tau$ is disconnected). If $\Sigma$ is not pinched at any $\tau\in\Sigma$, then we say that $\Sigma$ is \textbf{unpinched}. For example, a $3$-fan comprised of two maximal cones meeting along a ray is pinched at the ray, but a $3$-fan comprised of two maximal cones meeting along a $2$-cone is unpinched.

We are now ready to state our characterization of Lorentzian tropical fans.

\begin{theorem}\label{thm:characterization}
A quasiprojective tropical $d$-fan $\bSigma(u)$ is Lorentzian if and only if
\begin{enumerate}
\item[\mylabel{A}{\textnormal{(A)}}] $\Sigma$ is unpinched, and 
\item[\mylabel{B}{\textnormal{(B)}}]  $\bSigma(u)^\tau$ is Lorentzian for every $\tau\in\Sigma(d-2)$.
\end{enumerate}
\end{theorem}

Recall that stars of quasiprojective fans are quasiprojective and the Lorentzian condition is vacuous for stars of dimension at most one. Thus, \ref{B} is equivalent to imposing that
\begin{align*}
D(\Sigma^\tau)\times D(\Sigma^\tau)&\rightarrow\R\\
(D_1,D_2)&\mapsto \deg_{\bSigma(u)^\tau}(D_1D_2)
\end{align*}
has exactly one positive eigenvalue for every $\tau\in\Sigma(d-2)$, which is also equivalent to the condition that $\Vol_{\bSigma(u)^\tau}$ has exactly one positive eigenvalue for every $\tau\in\Sigma(d-2)$. We emphasize, again, that these conditions are independent of $u$. Interestingly, we note that this characterization of Lorentzian fans does not require one to have any knowledge about convex divisors on $\Sigma$ or its star fans.

Before proving Theorem~\ref{thm:characterization}, we discuss a few preliminary results. The first result shows that the unpinched condition is necessary for a tropical fan to be Lorentzian.

\begin{lemma}\label{lem:unpinched}
Let $\bSigma(u)$ be a quasiprojective tropical $d$-fan with $d\geq 2$. If $\Sigma\setminus\{0\}$ is disconnected, then $\bSigma(u)$ is not Lorentzian.
\end{lemma}

\begin{proof}
Assume that $\Sigma\setminus\{0\}$ is disconnected and, using this assumption, choose $d$-fans $\Sigma_1,\Sigma_2\subseteq\Sigma$ such that $\Sigma_1\cup\Sigma_2=\Sigma$ and $\Sigma_1\cap\Sigma_2=\{0\}$. Set $\omega_i=\omega|_{\Sigma_i(d)}$ and $u_i=u|_{\Sigma_i(1)}$. Note that 
\[
K(\Sigma)=K(\Sigma_1)\oplus K(\Sigma_2)\subseteq \R^{\Sigma_1(1)}\oplus\R^{\Sigma_2(1)}=\R^{\Sigma(1)}.
\]
Given $D_1\in K(\Sigma_1)$ and $D_2\in K(\Sigma_2)$, it follows that $\tilde D_1=(D_1,0)$ and $\tilde D_2=(0,D_2)$ are in $\overline K(\Sigma)$. Noting that $D_{u,\rho}\cdot D_{u,\eta}\cdot\omega=0\in MW_{d-2}(\Sigma,\omega)$ for any $\rho\in\Sigma_1(1)$ and $\eta\in\Sigma_2(1)$, it follows that
\[
\deg_{\bSigma}\big(\tilde D_1^k\tilde D_2^{d-k}\big)=\begin{cases}
\deg_{\bSigma_1}(D_1^d) & k=d,\\
\deg_{\bSigma_2}(D_2^d) & k=0,\\
0 & k=1,\dots,d-1.
\end{cases}
\]
Since $D_i\in K(\Sigma_i)$, \eqref{eq:positivity} implies that the first two cases are positive, and since $d\geq 2$, the sequence 
\[
\big(\deg_{\bSigma}\big(\tilde D_1^k\tilde D_2^{d-k}\big)\big)_{k=0}^d
\] 
is not unimodal. Thus, Proposition~\ref{prop:tropicalfanAF} implies that $\bSigma(u)$ is not Lorentzian.
\end{proof}

The next preparatory result was proved by Br\"and\'en and Leake \cite{BrandenLeake}. To state it, we require a few additional notions. The \textbf{lineality space} of an open convex cone $C\subseteq \R^n$ is defined to be $L_C=\overline C\cap -\overline C$; in other words, it is the largest subspace contained in the closure of $C$. We say that an open convex cone $C\subseteq\R^n$ is \textbf{effective} if $C=C\cap\R_{>0}^n+L_C$. A $k\times k$ matrix $M$ is said to be \textbf{irreducible} if its adjacency graph---the undirected graph on $k$ labeled vertices with an edge between the $i$th and $j$th vertex whenever the $(i,j)$ entry of $M$ is nonzero---is connected. Br\"and\'en and Leake proved the following result.

\begin{lemma}[\cite{BrandenLeake}, Proposition 2.4]\label{lem:BrandenLeake}
Let $f\in\R[x_1,\dots,x_n]$ be a homogeneous polynomial of degree $d\geq 3$, and let $C$ be an open, convex, and effective cone in $\R^n$. If 
\begin{enumerate}
\item $f(x+w)=f(x)$ for all $x\in\R^n$ and $w\in L_C$,
\item $\partial_{v_1}\cdots\partial_{v_d}f>0$ for all $v_1,\dots,v_d\in C$,
\item the Hessian of $\partial_{v_1}\cdots\partial_{v_{d-2}}f$ is irreducible and its off-diagonal entries are nonnegative for all $v_1,\dots,v_{d-2}\in C$, and
\item $\partial_if$ is $C$-Lorentzian for all $i=1,\dots, n$,
\end{enumerate}
then $f$ is $C$-Lorentzian.
\end{lemma}

In order to apply Lemma~\ref{lem:BrandenLeake} in our setting, the following result is necessary.

\begin{lemma}\label{lem:effective}
If $\bSigma(u)$ is a quasiprojective tropical $d$-fan, then $K_u(\Sigma)$ is effective.
\end{lemma}

\begin{proof}
We first prove that the lineality space $L_{K_u(\Sigma)}$ is the image of the map $V^\vee\rightarrow \R^{\Sigma(1)}$. 

To see that the image of $V^\vee\rightarrow \R^{\Sigma(1)}$ is contained in $L_{K_u(\Sigma)}$, let $\phi\in V^\vee$ and observe that
\[
D_u(z_u(\phi))=0=D_u(-z_u(\phi)).
\] 
Since $0\in \overline K(\Sigma)$, it follows that $z_u(\phi)\in L_{K_u(\Sigma)}$.

Conversely, suppose that $z\in L_{K_u(\Sigma)}$. By definition, this means that both $D_u(z)$ and $-D_u(z)$ are convex divisors. Thus, we can represent $D_u(z)=[\phi]$ and $-D_u(z)=[\phi']$ for some nonnegative $\phi,\phi'\in PL(\Sigma)$. Adding these nonnegative functions, we obtain a nonnegative function representing $0\in D(\Sigma)$; in other words, a nonnegative linear function on $\Sigma$. On the other hand, using the assumption that $\Sigma$ is quasiprojective and $\omega$ is positive, we see that $D^{d-1}\cdot\omega$ is a positive Minkowski $1$-weight for any $D\in K(\Sigma)$, from which it follows that $0$ can be written as a positive linear combination of the markings. In particular, this implies that the only nonnegative linear function on $\Sigma$ is the zero function, from which we see that $\phi$ and $\phi'$ must both be the zero functions. Thus, $D_u(z)=0$, and it follows that $z=z_u(\phi)$ for some $\phi\in V^\vee$.

We now argue that $K_u(\Sigma)$ is effective; more precisely, we show that
\[
K_u(\Sigma)=K_u(\Sigma)\cap \R_{> 0}^{\Sigma(1)}+L_{K_u(\Sigma)}.
\]
Suppose $z\in K_u(\Sigma)$. Then $D_u(z)\in K(\Sigma)$, so there exists a linear function $\phi\in V^\vee$ such that
\[
z-z_u(\phi)\in\R_{>0}^{\Sigma(1)}.
\]
Since $z_u(\phi)\in L_{K_u(\Sigma)}$, it follows that
\[
z=(z-z_u(\phi))+z_u(\phi)\in K_u(\Sigma)\cap \R_{> 0}^{\Sigma(1)}+L_{K_u(\Sigma)}.
\]
Conversely, suppose that $z\in K_u(\Sigma)\cap \R_{> 0}^{\Sigma(1)}+L_{K_u(\Sigma)}$. Then $z=z'+z_u(\phi)$ for some $z'\in K_u(\Sigma)$ and $\phi\in V^\vee$. Translating by $z_u(\phi)$ does not affect membership in $K_u(\Sigma)$, so $z\in K_u(\Sigma)$.
\end{proof}

We are now prepared to prove our characterization of Lorentzian tropical fans.

\begin{proof}[Proof of Theorem~\ref{thm:characterization}]
Let $\bSigma(u)$ be a quasiprojective tropical $d$-fan. Assume, first, that $\bSigma(u)$ is Lorentzian. It follows that $\bSigma(u)^\tau$ is Lorentzian for every $\tau\in\Sigma$, implying \ref{B}. Furthermore, Lemma~\ref{lem:unpinched} implies that $|\Sigma^\tau|\setminus\{0\}$ is connected for every $\tau\in \Sigma(k)$ with $0\leq k\leq d-2$, showing that $\Sigma$ is unpinched, proving \ref{A}. 

Conversely, assume that $\bSigma(u)$ satisfies \ref{A} and \ref{B}; we prove that $\bSigma(u)$ is Lorentzian. By Proposition~\ref{prop:lorentzians}, it suffices to show that $\Vol_{\bSigma(u)^\tau}$ is $K_{u^\tau}(\Sigma^\tau)$-Lorentzian for every $\tau\in \Sigma(k)$ with $0\leq k\leq d-2$. We accomplish this by induction on the codimension of $\tau$. The base case $\dim(\tau)=d-2$ follows from \ref{B}. For the induction step, assume that $\Vol_{\bSigma(u)^\tau}$ is $K_{u^\tau}(\Sigma^\tau)$-Lorentzian for every $\tau\in \Sigma(k)$ with $\ell+1\leq k\leq d-2$ for some $\ell\in\{0,\dots,d-3\}$, we must show that $\Vol_{\bSigma(u)^\tau}$ is $K_{u^\tau}(\Sigma^\tau)$-Lorentzian for every $\tau\in \Sigma(\ell)$. For the purposes of this argument, and to ease notation, it suffices to assume that $\Sigma^\tau=\Sigma$.

To prove that $\Vol_{\bSigma(u)}$ is $K_{u}(\Sigma)$-Lorentzian, we prove the four conditions of Lemma~\ref{lem:BrandenLeake}. Condition (1) follows from the fact that every element of $L_{K_{u}(\Sigma)}$ is $z_{u}(\phi)$ for some $\phi\in V^\vee$ (as was shown in the proof of Lemma~\ref{lem:effective}), along with the fact that $D_u(z)\in D(\Sigma)$ is invariant under shifting the argument by $z_u(\phi)$. Condition (2) follows from \eqref{eq:positivity} and \eqref{eq:partials}. Thus, it remains to verify (3) and (4); to do so, we study derivatives of the volume polynomial.

For any two rays $\rho,\eta\in\Sigma(1)$, we have 
\[
\partial_\rho\Vol_{\bSigma(u)}=(d-\ell)\cdot \deg_{\bSigma(u)}(D_{u,\rho} \cdot D_{u}(z)^{d-\ell-1}),
\]
and
\[
\partial_\eta\partial_\rho\Vol_{\bSigma(u)}=(d-\ell)(d-\ell-1)\cdot\deg_{\bSigma(u)}(D_{u,\eta} D_{u,\rho} \cdot D_{u}(z)^{d-\ell-2}).
\]
Applying \eqref{eq:reduction} to the right-hand sides of the above equations, we obtain
\begin{equation}\label{eq:firstderivatives}
\partial_\rho\Vol_{\bSigma(u)}=(d-\ell)\cdot \deg_{\bSigma(u)^\rho}\big(\overline{D_{u}(z)}^{d-\ell-1}\big),
\end{equation}
and, assuming that $\eta$ and $\rho$ are distinct rays, we have
\begin{equation}\label{eq:secondderivatives}
\partial_\eta\partial_\rho\Vol_{\bSigma(u)}=\begin{cases}
(d-\ell)(d-\ell-1)\cdot\deg_{\bSigma(u)^\pi}\big(\overline{D_u(z)}^{d-\ell-2}\big) &\text{if }\{\rho,\eta\}=\pi(1)\\
0 & \text{else.}
\end{cases}
\end{equation}

To prove Condition (3), let $z_1,\dots,z_{d-\ell-2}\in K_u(\Sigma)$. It follows from \eqref{eq:secondderivatives} that the off-diagonal $(\rho,\eta)$-entry of the Hessian of $\partial_{z_1}\cdots\partial_{z_{d-\ell-2}}\Vol_{\bSigma(u)}$ is either $0$ if $\rho$ and $\eta$ do not form the rays of a $2$-cone in $\Sigma$, or it is
\[
(d-\ell)!\cdot\deg_{\bSigma(u)^\pi}\big(\overline{D_u(z_1)}\cdots \overline{D_u(z_{d-\ell-2}})\big)
\]
if $\{\rho,\eta\}=\pi(1)$ for some $\pi\in\Sigma(2)$. Since $\overline D\in K_{u^\pi}(\Sigma^\pi)$ for every $D\in K(\Sigma)$, it follows from \eqref{eq:positivity} that all of the off-diagonal entries of the Hessian are nonnegative, and the entries indexed by pairs $(\rho,\eta)$ that form the rays of a $2$-cone in $\Sigma$ are strictly positive. Condition \ref{A} ensures that $|\Sigma|\setminus\{0\}$ is connected, from which the above positivity in off-diagonal terms implies that the Hessian of $\partial_{z_1}\cdots\partial_{z_{d-\ell-2}}\Vol_{\bSigma(u)}$ is irreducible, verifying (3).

Finally, to prove (4), we must argue that $\partial_\rho\Vol_{\bSigma(u)}$ is $K_u(\Sigma)$-Lorentzian. Property (P) follows from \eqref{eq:positivity}, so it remains to prove Property (H), which is equivalent to the assertion that, for any $z_3,\dots,z_{d-\ell-1}\in K_u(\Sigma)$, the quadratic form
\begin{equation}\label{eq:matrix1}
\partial_{z_3}\cdots\partial_{z_{d-\ell-1}}\partial_\rho\Vol_{\bSigma(u)}
\end{equation}
has exactly one positive eigenvalue. The induction hypothesis implies that, given any $ z_3,\dots, z_{d-\ell-1}\in K_{u^\rho}(\Sigma^\rho)$, the quadratic form
\begin{equation}\label{eq:matrix2}
\partial_{ z_3}\cdots\partial_{ z_{d-\ell-1}}\Vol_{\bSigma(u)^\rho}
\end{equation}
has exactly one positive eigenvalue. Note that \eqref{eq:matrix1} and \eqref{eq:matrix2} can be related using \eqref{eq:firstderivatives}, which implies that, for any $z_3,\dots,z_{d-\ell-1}\in\R^{\Sigma(1)}$, we have
\begin{equation}
\partial_{z_3}\cdots\partial_{z_{d-\ell-1}}\partial_\rho\Vol_{\bSigma(u)}=\frac{(d-\ell)!}{2}\cdot\partial_{\overline z_3}\cdots\partial_{\overline z_{d-\ell-1}}\Vol_{\bSigma(u)^\rho},
\end{equation}
where $\overline z_i\in \R^{\Sigma^\rho(1)}$ is any vector such that $D(\overline z_i)=\overline{D(z_i)}$. If we choose $\phi\in V^\vee$ such that $\phi(u_\rho)=-1$, then for any $z\in\R^{\Sigma(1)}$, we can define $\overline z\in\R^{\Sigma^\rho(1)}$ by 
\begin{equation}\label{eq:changeofvariables}
\overline z_{\overline \eta}=z_\eta+z_\rho\phi(u_\eta).
\end{equation}
It then follows that \eqref{eq:matrix2} specializes to \eqref{eq:matrix1} under the linear change of variables \eqref{eq:changeofvariables}:
\[
\partial_{z_3}\cdots\partial_{z_{d-\ell-1}}\partial_\rho\Vol_{\bSigma(u)}(z)=\frac{(d-\ell)!}{2}\partial_{\overline z_3}\cdots\partial_{\overline z_{d-\ell-1}}\Vol_{\bSigma(u)^\rho}(\overline z)
\]
Extending the change of variables to an invertible change of variables by introducing an extraneous variable $\overline z_{\overline\rho}=z_\rho$, it follows that the two quadratic forms \eqref{eq:matrix1} and \eqref{eq:matrix2} have the same signature, so the fact that the \eqref{eq:matrix2} has exactly one positive eigenvalue implies that \eqref{eq:matrix1} also has exactly one positive eigenvalue.
\end{proof}

In order to apply Theorem~\ref{thm:characterization} in practice, it is necessary to understand the volume polynomials of the two-dimensional star fans of $\Sigma$. As it turns out, these two-dimensional volume polynomials can be computed in full generality, as we describe in the next result, which then makes the applicability of Theorem~\ref{thm:characterization} completely concrete.

\begin{proposition}\label{prop:2dvolume}
If $\bSigma(u)$ is a tropical $2$-fan, then
\[
\Vol_{\bSigma(u)}=\sum_{\sigma\in\Sigma(2)\atop \sigma(1)=\{\rho_1,\rho_2\}}2 \omega(\sigma)z_{\rho_1} z_{\rho_2}-\sum_{\rho\in\Sigma(1)}a_{\rho} z_\rho^2
\]
where $a_\rho$ is determined by the formula
\[
\sum_{\sigma\in\Sigma(1)\atop\rho\prec\sigma}\omega(\sigma)u_{\sigma/\rho}=a_{\rho} u_\rho.
\]
\end{proposition}

\begin{proof}
By the Chow-theoretic definition of the volume polynomial
\[
\Vol_{\bSigma(u)}(z)=f_{\bSigma(u)}(X_u(z)^2),
\]
we must verify that
\begin{enumerate}
\item $f_{\bSigma(u)}(X_{u,\eta_1}X_{u,\eta_1})=0$ for all $\eta_1,\eta_2\in\Sigma(1)$ that do not lie on a common cone,
\item $f_{\bSigma(u)}(X_{u,\eta_1}X_{u,\eta_1})=\omega(\sigma)$ for all $\sigma\in\Sigma(2)$ with rays $\eta_1$ and $\eta_2$, and
\item $f_{\bSigma(u)}(X_{u,\rho}^2)=-a_\rho$ for all $\rho\in\Sigma(1)$.
\end{enumerate}
The first follows immediately from definition of $A^\bullet(\Sigma,u)$, and the second is essentially the definition of $f_{\bSigma(u)}$. To prove the third, choose $\phi\in V^\vee$ such that $\phi(u_\rho)=-1$, and write
\[
X_{u,\rho}=\sum_{\eta\in\Sigma(1)\atop\eta\neq\rho}\phi(u_\eta) X_{u,\eta}\in A^1(\Sigma,u).
\]
We then compute
\begin{align*}
f_{\bSigma(u)}(X_{u,\rho}^2)&=f_{\bSigma(u)}\Big(X_{u,\rho}\sum_{\eta\in\Sigma(1)\atop\eta\neq\rho}\phi(u_\eta) X_\eta\Big)\\
&=\sum_{\sigma\in\Sigma(2)\atop \rho\prec\sigma}\phi(u_{\sigma/\rho})\omega(\sigma)\\
&=\phi(a_{\rho} u_\rho)\\
&=-a_{\rho}.\qedhere
\end{align*}
\end{proof}

An important class of Lorentzian fans is the class of Bergman fans of matroids (which we do not define here). The proof of the following result is essentially contained in \cite[Section~6]{NowakOMelvenyRoss} (see also \cite{BackmanEurSimpson} and \cite{BrandenLeake} for similar treatments of Condition \ref{B}). 

\begin{theorem}\label{thm:Bergman}
Bergman fans of matroids are Lorentzian.
\end{theorem}

\begin{proof}
By Theorem~\ref{thm:characterization}, it suffices to prove that Bergman fans satisfy \ref{A} and \ref{B}. Property \ref{A} is \cite[Lemma~6.4]{NowakOMelvenyRoss} and Property \ref{B} follows from \cite[Lemmas 6.5 and 6.6]{NowakOMelvenyRoss}.
\end{proof}

\subsection{Disentangling geometry from analysis}

We close this section with reflections on the developments that led to Theorem~\ref{thm:characterization}. As we emphasized in the introduction, the ideas in this paper are heavily influenced by the AF inequalities, and if we specialize Theorem~\ref{thm:characterization} to the setting of complete fans, it essentially tells us that the AF inequalities (for simple strongly isomorphic polytopes) can be deduced from the two-dimensional AF inequalities with only a little bit of geometric insight---basically, the only geometric insight required is the reduction \eqref{eq:reduction}. This is quite different in style from Aleksandrov's original proof of the AF inequalities in \cite{Aleksandrov} (see \cite{Schneider} for a reproduction in English), where the geometric and analytic arguments are closely intertwined, and it was only very recently that the key geometric insight \eqref{eq:reduction} was disentangled from the rest of the analytic arguments. We view the key first step in this disentanglement to be the recent proof of the AF inequalities by Cordero-Erausquin, Klartag, Merigot, and Santambrogio \cite{OneMoreProof}, where it became clear that \eqref{eq:reduction} was the only essential geometric input required to deduce the AF inequalities from the two-dimensional setting. However, the analytic arguments in \cite{OneMoreProof} were still quite complex, requiring a doubly-indexed induction, and we view the second key step in the disentanglement to be the recent work of Br\"and\'en and Leake \cite{BrandenLeake}, which, by changing the focus away from the AF inequalities and instead focusing on the Lorentzian property,  significantly simplified and clarified the analytic arguments in \cite{OneMoreProof}. Now that the disentanglement seems complete, it will be interesting to see what other settings one might be able to apply these ideas.

\section{Basic properties of Lorentzian fans}\label{sec:properties}

In this short section, we collect several useful ways to construct new Lorentzian fans from old ones. Our first result shows that the Lorentzian property is compatible with tropical products. Before stating it, we recall that the product of two fans is simply the collection of pairwise products of their cones:
\[
\Sigma_1\times\Sigma_2=\{\sigma_1\times\sigma_2\mid\sigma_i\in\Sigma_i\}.
\]
If $\Sigma_i$ is a simplicial $d_i$-fan and $\omega_i\in MW_{d_i}(\Sigma_i)$ for $i=1,2$, we naturally obtain a Minkowski weight $\omega_1\times\omega_2\in MW_{d_1+d_2}(\Sigma_1\times\Sigma_2)$ defined by
\[
(\omega_1\times\omega_2)(\sigma_1\times\sigma_2)=\omega_1(\sigma_1)\omega_2(\sigma_2)\;\;\;\text{ for all }\;\;\;\sigma_1\in\Sigma_1(d_1),\;\sigma_2\in\Sigma_2(d_2).
\]
Thus, if $\bSigma_1=(\Sigma_1,\omega_1)$ and $\bSigma_2=(\Sigma_2,\omega_2)$ are tropical fans, we can define their product in the setting of tropical fans as
\[
\bSigma_1\times\bSigma_2=(\Sigma_1\times\Sigma_2,\omega_1\times\omega_2).
\]

\begin{proposition}
If $\bSigma_1$ and $\bSigma_2$ are tropical fans, then $\bSigma_1\times\bSigma_2$ is Lorentzian if and only if $\bSigma_1$ and $\bSigma_2$ are Lorentzian.
\end{proposition}

\begin{proof}
The key observation is that, for $\tau_1\in\Sigma_1$ and $\tau_2\in\Sigma_2$, we have
\begin{equation}\label{eq:products}
(\bSigma_1\times\bSigma_2)^{\tau_1\times\tau_2}=\bSigma_1^{\tau_1}\times\bSigma_2^{\tau_2}.
\end{equation}
One readily checks that $K(\Sigma_1\times\Sigma_2)=K(\Sigma_1)\times K(\Sigma_2)$, so the product is quasiprojective if and only if each factor is quasiprojective. By taking $\tau_1$ or $\tau_2$ to be maximal cones, we see from \eqref{eq:products} that the star fans of $\bSigma_1\times\bSigma_2$ are a superset of the star fans of $\bSigma_1$ and $\bSigma_2$, and since the Lorentzian condition is a stipulation on all of the star fans, this implies that $\bSigma_1$ and $\bSigma_2$ are Lorentzian whenever $\bSigma_1\times\bSigma_2$ is Lorentzian. To prove the converse, it suffices, by Theorem~\ref{thm:characterization}, to show that \ref{A} and \ref{B} are satisfied for all stars $(\bSigma_1\times\bSigma_2)^{\tau_1\times\tau_2}$ where neither $\tau_1$ nor $\tau_2$ is maximal. For \ref{A}, this amounts to the observation that any product of two positive dimensional fans remains connected after removing the origin; thus, $\bSigma_1\times\bSigma_2$ satisfies \ref{A} if both $\bSigma_1$ and $\bSigma_2$ satisfy \ref{A}. For \ref{B}, notice that the only case to consider is when $\dim(\tau_i)=\dim(\Sigma_i)-1$ for $i=1,2$. In this case, one readily verifies that, for any markings $u_i$ on $\Sigma_i$, we have
\[
\Vol_{(\bSigma_1(u_1)\times\bSigma_2(u_2))^{\tau_1\times\tau_2}}=\Vol_{\bSigma_1(u_1)^{\tau_1}}\Vol_{\bSigma_2(u_2)^{\tau_2}}.
\]
Since a product of linear forms in disjoint variables always has one positive eigenvalue (because $\ell_1\ell_2=\frac{1}{4}(\ell_1+\ell_2)^2-\frac{1}{4}(\ell_1-\ell_2)^2$), it follows that $\bSigma_1\times\bSigma_2$ satisfies \ref{B} if both $\bSigma_1$ and $\bSigma_2$ satisfy \ref{B}, completing the proof.
\end{proof}

To state the next property, we describe first describe an action of strictly convex divisors on tropical fans. Suppose that $\bSigma(u)$ is a tropical $d$-fan and $D\in K(\Sigma)$ is a strictly convex divisor. We define $D\cdot \bSigma(u)$ to be the tropical $(d-1)$-fan
\[
D\cdot\bSigma(u)=(\Sigma[d-1],u,D\cdot\omega).
\]
We now show that this action is compatible with the Lorentzian property.

\begin{proposition}\label{prop:actlorentzian}
If $\bSigma(u)$ is Lorentzian and $D\in K(\Sigma)$, then $D\cdot\bSigma(u)$ is Lorentzian.
\end{proposition}

\begin{proof}
Suppose that $\bSigma(u)$ is a Lorentzian $d$-fan and $D\in K(\Sigma)$. First of all, we observe that $D$ descends to an element of $K(\Sigma[d-1])$, so $\Sigma[d-1]$ is quasiprojective. Additionally, for any $\tau\in\Sigma[d-1]$, note that the star of the codimension-one skeleton is the codimension-one skeleton of the star:
\[
\Sigma[d-1]^\tau=\Sigma^\tau[d^\tau-1].
\]
It then follows from the definition of the $D(\Sigma)$-action on Minkowski weights that
\[
\deg_{D\cdot \bSigma(u)^\tau}(D_1\cdots D_{d^\tau-1})=\deg_{\bSigma(u)^\tau}(D_1\cdots D_{d^\tau-1}\cdot\overline D).
\]
Since $D\in K(\Sigma)$, it follows that $\overline D\in K(\Sigma^\tau)$, and thus, the quadratic forms in Definition~\ref{def:lorentzian} associated to $D\cdot \bSigma(u)^\tau$ are all special cases of the quadratic forms associated to $\bSigma(u)^\tau$, so the Lorentizan property for $D\cdot \bSigma(u)$ follows from the Lorentzian property for $\bSigma(u)$.
\end{proof}

Our final result in this section regards tropical modifications, which are yet another way of producing new tropical fans from old ones. The input of a tropical modification consists of a tropical $d$-fan $\bSigma=(\Sigma,\omega)$ in $V$ and a strictly convex piecewise-linear function $\phi\in PL(\Sigma)$; the output is a new tropical $d$-fan $M_\phi\bSigma$ in $V\times\R$. More precisely, the \textbf{tropical modification of $\bSigma$ along $\phi$} is defined by
\[
M_\phi\bSigma=(\Sigma_\phi,\omega_\phi)
\]
where the fan $\Sigma_\phi$ has cones
\[
\Sigma_\phi=\{\gamma_\phi(\tau)\mid\tau\in\Sigma\}\cup \{\tau_{(0,-1)}\mid \tau\in\Sigma\setminus\Sigma(d)\}
\]
where $\gamma_\phi(\tau)=\{(x,\phi(x))\in N_\R\times \R\mid x\in\tau\}$ and $\tau_{(0,-1)}=\{u+\lambda(0,-1)\mid u\in\tau,\lambda\in\R_{\geq 0}\}$, and the Minkowski weight $\omega_\phi$ is defined by
\[
\omega_\phi(\sigma)=\begin{cases}
\omega(\tau) &\text{if }\sigma=\gamma_\phi(\tau)\text{ for some }\tau\in\Sigma(d);\\
(\phi\cdot\omega)(\tau) &\text{if }\sigma=\tau_{(0,-1)}\text{ for some }\tau\in\Sigma(d-1).
\end{cases}
\]
It is straightforward to check that $\omega_\phi\in MW_d(\Sigma_\phi)$ is a positive Minkowski weight.

We note that the graph $\gamma_\phi(\Sigma)$ of $\Sigma$ under the function $\phi$ is contained within $\Sigma_\phi$, but is not a tropical fan with respect to the induced weights on the maximal cones---the tropical modification $M_\phi\bSigma$ is, in some sense, the unique minimal way of appending weighted cones to $\gamma_\phi(\Sigma)$ to obtain a tropical fan. Whenever $\phi_1$ and $\phi_2$ differ by a linear function $f\in L(\Sigma)$, we note that $M_{\phi_1}\bSigma$ and $M_{\phi_2}\bSigma$ are isomorphic under the linear automorphism
\begin{align*}
V\times\R&\rightarrow V\times\R\\
(v,w)&\mapsto (v,w+f(v)),
\end{align*}
so, up to linear isomorphism, we view the tropical modification as happening along a strictly convex divisor $D=[\phi]\in K(\Sigma)$ and we denote it by $M_D(\bSigma)$. The following result shows that the Lorentzian property is compatible with tropical modifications.

\begin{proposition}
If $\bSigma$ is a tropical fan and $D\in K(\Sigma)$, then $M_D\bSigma$ is Lorentzian if and only if $\bSigma$ is Lorentzian.
\end{proposition}

\begin{proof}
Choose a representative $\phi$ for $D$ and a marking $u$ on $\Sigma$, and let $u_\phi$ be the associated marking on $\Sigma_\phi$:
\[
u_\rho=\{(u_\rho,\phi(u_\rho)\}\cup\{(0,-1)\}.
\] 
There are two key observations, both of which are straightforward to verify from the definitions. The first is that there is a natural isomorphism
\begin{align*}
A^\bullet(\Sigma,u)&\stackrel{\cong}{\longrightarrow}A^\bullet(\Sigma_\phi,u_\phi)\\
X_{u,\rho}&\longmapsto X_{u_\phi,\gamma_\phi(\rho)},
\end{align*}
that commutes with the maps $f_{\Sigma,u,\omega}$ and $f_{\Sigma_\phi,u_\phi,\omega_\phi}$ and identifies the cones of strictly convex divisors (via the isomorphisms $D(\Sigma)\cong A^1(\Sigma,u)$ and $D(\Sigma_\phi)\cong A^1(\Sigma_\phi,u_\phi)$). In particular, this shows that the quadratic forms associated to $(\Sigma,u,\omega)$ have exactly one positive eigenvalue if and only if the same holds for $(\Sigma_\phi,u_\phi,\omega_\phi)$. This conclusion does not apply to the quadratic forms associated to nontrivial stars of $\bSigma$ and $M_D\bSigma$; for that, we need the second key observation, which is that, for any $\tau\in \Sigma_D(k)$, 
\[
(M_D\bSigma)^\tau=\begin{cases}
(D\cdot \bSigma)^{\tau'} &\text{if }\tau=\tau'_{(0,-1)}\text{ for some }\tau'\in\Sigma(k-1),\\
M_{\overline D}(\bSigma^\tau) &\text{if } \tau=\gamma_\phi(\tau')\text{ for some }\tau'\in\Sigma(k).
\end{cases}
\]

Suppose that $\bSigma$ is Lorentzian. By the second observation, every star fan of $M_D\bSigma$ is either a star of $D\cdot\bSigma$ or a tropical modification of a star of $\bSigma$. In the first case, the associated quadratic forms have exactly one positive eigenvalue by Proposition~\ref{prop:actlorentzian}, while in the latter case, the associated quadratic forms have exactly one positive eigenvalue by the first observation and the assumption that $\bSigma$ is Lorentzian. Thus, $M_D\bSigma$ is Lorentzian.

Conversely, suppose that $M_D\bSigma$ is Lorentzian. By the second observation, every star of $\bSigma$ can be tropically modified by $\overline D$ to become a star of $M_D\bSigma$. The assumption that $M_D\bSigma$ is Lorentzian along with the first observation then imply that $\bSigma$ is Lorentzian.
\end{proof}

%
%

It would be natural, at this point, to continue our list of Lorentzian-preserving constructions by showing that the Lorentzian property is preserved by stellar subdivisions (it is!). We choose not to do that here, however, because our ultimate aim is to prove a stronger result: that the Lorentzian property does not actually depend on the fan structure at all, it only depends on the underlying \emph{tropical variety}. In order to state this result precisely, we take a detour in the next section to introduce the theory of \emph{tropical fan varieties}, before returning to the study of stellar subdivisions and the independence of the Lorentzian property on the fan structure in Section~\ref{sec:lorentzianfansets}.

\section{Minkowski weights on fan sets and tropical fan varieties}\label{sec:variety}

In this section, we introduce a general theory of Minkowski weights on fan sets. Because we do not require the supports of our Minkowski weights to be rational fans, the notions developed here generalize the \emph{affine tropical cycles} introduced by Allerman and Rau \cite{AllermanRau}, which built on prior work of Gathmann, Kerber, and Markwig \cite{GathmannKerberMarkwig}. Many of the arguments from those earlier papers carry over in a more-or-less straightforward way to the general setting, but since the development here seems sufficiently novel, we carefully spell out the details of the most important aspects.

\subsection{Minkowski weights on vector spaces} 

Let $V$ be an $n$-dimensional real vector space. A \textbf{marked Minkowski $d$-weight on $V$} is a triple $(\Sigma,u,\omega)$ where 
\begin{itemize}
\item $\Sigma$ is a simplicial $d$-fan, 
\item $u$ is a marking of $\Sigma$, and 
\item $\omega\in MW_d(\Sigma,u)$ is a Minkowski $d$-weight on the marked fan $(\Sigma,u)$. 
\end{itemize}
We define the \textbf{support} of a marked Minkowski $d$-weight on $V$ to be the support of the subfan where $\omega$ is nonzero:
\[
|(\Sigma,u,\omega)|=\bigcup_{\sigma\in\Sigma(d)\atop \omega(\sigma)\neq 0}\sigma.
\]
Two marked Minkowski $d$-weights $(\Sigma_1,u_1,\omega_1)$ and $(\Sigma_2,u_2,\omega_2)$ are \textbf{equivalent} if
\[
|(\Sigma_1,u_1,\omega_1)|=|(\Sigma_2,u_2,\omega_2)|
\]
and, for every pair of cones $\sigma_1\in\Sigma_1(d)$ and $\sigma_2\in\Sigma_2(d)$ such that
\begin{enumerate}
\item[(i)] $\sigma_i\subseteq |(\Sigma_i,u_i,\omega_i)|$ for $i=1,2$ and
\item[(ii)] $\sigma_1^\circ\cap\sigma_2^\circ\neq\emptyset$,
\end{enumerate}
we have
\begin{equation}\label{eq:volumeratio}
\omega_1(\sigma_1)=\frac{\vol(\sigma_1,u_1)}{\vol(\sigma_2,u_2)}\omega_2(\sigma_2),
\end{equation}
where $\frac{\vol(\sigma_1,u_1)}{\vol(\sigma_2,u_2)}$ denotes the absolute value of the determinant of any linear tranformation of $V_{\sigma_1}=V_{\sigma_2}$ that takes the markings of $\sigma_2$ to the markings of $\sigma_1$. Geometrically, notice that $\frac{\vol(\sigma_1,u_1)}{\vol(\sigma_2,u_2)}$ is the just ratio of the volumes of the parallelepipeds determined by the markings on $\sigma_1$ and $\sigma_2$, respectively, justifying the notation. While the volumes, themselves, depend on a metric, the ratio does not, and the condition \eqref{eq:volumeratio} can interpreted as the imposition that the weights are equal after normalizing by the volume of the associated parallelepipeds.

We define a \textbf{Minkowski $d$-weight on $V$} to be an equivalence class of marked Minkowski $d$-weights on $V$, and we denote the set of Minkowski $d$-weights on $V$ by $MW_d(V)$. We aim to endow $MW_d(V)$ with the structure of a real vector space; the following result is the key observation that we need.

\begin{proposition}\label{prop:fansets}
If $(\Sigma,u,\omega)$ is a marked Minkowski $d$-weight and $(\Sigma',u')$ is any marked simplicial fan with $|\Sigma'|=|\Sigma|$, then there exists a unique $\omega'\in MW_d(\Sigma',u')$ such that $(\Sigma,u,\omega)$ and $(\Sigma',u',\omega')$ are equivalent.
\end{proposition}

\begin{proof}

Since any two simplicial fans with the same support admit a common refinement by a simplicial fan, it suffices to consider two cases:
\begin{enumerate}
\item $\Sigma'$ is a refinement of $\Sigma$, and
\item $\Sigma$ is a refinement of $\Sigma'$.
\end{enumerate}
The considerations in each of these cases is similar to the discussions in \cite{GathmannKerberMarkwig}.

If $\Sigma'$ refines $\Sigma$ and $\sigma'\in\Sigma'(d)$, then there is a unique $\sigma\in\Sigma(d)$ such that $\sigma'\subseteq\sigma$; define
\[
\omega'(\sigma')=\frac{\vol(\sigma',u')}{\vol(\sigma,u)}\omega(\sigma).
\]
In fact, we \emph{must} define $\omega'$ in this way in order for $(\Sigma',u',\omega')$ to be equivalent to $(\Sigma,u,\omega)$; what's not immediately apparent is that $\omega'$ is actually a Minkowski weight on $(\Sigma',u')$. To check that $\omega'\in MW_d(\Sigma',u')$, let $\tau'\in\Sigma'(d-1)$. We consider two cases:
\begin{enumerate}
\item[(1a)] $\tau'\subseteq\tau$ for some $\tau\in\Sigma(d-1)$, and
\item[(1b)] $\tau'\not\subseteq\tau$ for any $\tau\in\Sigma(d-1)$.
\end{enumerate}

In Case (1a), note that there is a natural bijection between $\{\sigma'\in\Sigma'(d)\mid\tau'\preceq\sigma'\}$ and $\{\sigma\in\Sigma(d)\mid\tau\preceq\sigma\}$. A linear algebra argument shows that
\begin{equation}\label{eq:bc1}
\frac{\vol(\sigma',u')}{\vol(\sigma,u)}u_{\sigma'\setminus\tau'}'=\frac{\vol(\tau',u')}{\vol(\tau,u)}u_{\sigma\setminus\tau} \mod V_\tau,
\end{equation}
and it follows that the balancing condition \eqref{eq:balancingcondition} holds at $\tau'$ if and only if it holds at $\tau$. 

In Case (1b), there is a unique $\sigma\in \Sigma(d)$ such that $\tau'\subseteq\sigma$ and there are exactly two cones $\sigma_1',\sigma_2'\in\Sigma'(d)$ with $\tau'\preceq\sigma_1',\sigma_2'$, both contained within $\sigma$. Another linear algebra argument implies that
\begin{equation}\label{eq:bc2}
\frac{\vol(\sigma_1',u')}{\vol(\sigma,u)}u_{\sigma_1'\setminus\tau'}'+\frac{\vol(\sigma_2',u')}{\vol(\sigma,u)}u_{\sigma_2'\setminus\tau'}'=0 \mod V_\tau,
\end{equation}
which, upon multiplying by $\omega(\sigma)$, is equivalent to the balancing condition \eqref{eq:balancingcondition} at $\tau'$.

Next, we turn to the case where $\Sigma$ refines $\Sigma'$. For each $\sigma'\in\Sigma'(d)$, choose some $\sigma\in\Sigma(d)$ such that $\sigma\subseteq\sigma'$, and define
\[
\omega'(\sigma')=\frac{\vol(\sigma',u')}{\vol(\sigma,u)}\omega(\sigma).
\]
Again, this is how we \emph{must} define $\omega'$ if we want $(\Sigma',u',\omega')$ to be equivalent to $(\Sigma,u,\omega)$. We first argue that the definition of $\omega'$ is independent of the choice of $\sigma$; to do so, it suffices to show that we get the same value for $\omega'(\sigma')$ if we choose $\sigma_1,\sigma_2\subseteq\sigma'$ such that $\sigma_1\cap\sigma_2=\tau\in\Sigma(d-1)$. Note that the balancing condition \eqref{eq:balancingcondition} at $\tau$ implies that
\begin{equation}\label{eq:bc3}
\omega(\sigma_1)u_{\sigma_1\setminus\tau}+\omega(\sigma_2)u_{\sigma_2\setminus\tau}=0 \mod V_\tau.
\end{equation}
As in \eqref{eq:bc2}, we have
\begin{equation}\label{eq:bc4}
\frac{\vol(\sigma_1,u)}{\vol(\sigma',u')}u_{\sigma_1\setminus\tau}+\frac{\vol(\sigma_2,u)}{\vol(\sigma',u')}u_{\sigma_2\setminus\tau}=0 \mod V_\tau
\end{equation}
Combining \eqref{eq:bc3} and \eqref{eq:bc4} implies that 
\[
\frac{\vol(\sigma',u')}{\vol(\sigma_1,u)}\omega(\sigma_1)=\frac{\vol(\sigma',u')}{\vol(\sigma_2,u)}\omega(\sigma_2),
\]
showing that $\omega'(\sigma')$ does not depend on the choice of $\sigma\subseteq\sigma'$

To check that $\omega'\in MW_d(\Sigma',u')$, let $\tau'\in\Sigma'(d-1)$. There must be some $\tau\in\Sigma(d-1)$ with $\tau\subseteq\tau'$, and the argument in (1a) shows that  the balancing condition \eqref{eq:balancingcondition} holds at $\tau'$ if and only if it holds at $\tau$, completing the proof of Case (2).
\end{proof}

We now describe the vector space structure on $MW_d(V)$.

\begin{proposition}
The set of Minkowski $d$-weights on $V$ is a real vector space. More precisely, given $\Omega_1,\Omega_2\in MW_d(V)$ and $\lambda\in\R$, we define $\lambda\Omega_1+\Omega_2\in MW_d(V)$ by 
\[
\lambda\Omega_1+\Omega_2=[\Sigma,u,\lambda\omega_1+\omega_2]
\]
where $(\Sigma,u,\omega_i)$ represents $\Omega_i$ for $i=1,2$.
\end{proposition}

\begin{proof}
Given two Minkowski $d$-weights $\Omega_1,\Omega_2\in MW_d(V)$, we first argue that they can be represented on a common marked fan $(\Sigma,u)$. To do so, we begin by constructing a fan $\Sigma$ whose support is $|\Omega_1|\cup|\Omega_2|$ and such that $\Sigma$ contains two subfans $\Sigma_1,\Sigma_2\subseteq\Sigma$ with $|\Sigma_i|=|\Omega_i|$ (see, for example, \cite[Construction 2.13]{AllermanRau}). Upon triangulating, if necessary, we may assume that $\Sigma$ is simplicial, and we can choose any marking $u$ of $\Sigma$. By Proposition~\ref{prop:fansets}, each $\Omega_i$ has a unique representative supported on $(\Sigma_i,u)$, which we can extend by zero to all of $\Sigma$. Thus, we may represent both $\Omega_1$ and $\Omega_2$ on a common marked simplicial fan $(\Sigma,u)$.

Next, we argue that the operation is well-defined. If $(\Sigma',u')$ is any other choice of marked simplicial fan on which we can represent $\Omega_1$ and $\Omega_2$, then Proposition~\ref{prop:fansets} ensures that there exists a unique $\omega_i'$ such that $(\Sigma,u,\omega_i)$ is equivalent to $(\Sigma',u',\omega_i')$. By definition of equivalence, we know that, for any $\sigma\in\Sigma(d)$ and $\sigma'\in\Sigma'(d)$ with $\sigma^\circ\cap\sigma'^\circ\neq\emptyset$, $\omega_i(\sigma)\neq 0$, and $\omega_i'(\sigma')\neq 0$, we have
\[
\omega_i(\sigma)=\frac{\vol(\sigma,u)}{\vol(\sigma',u')}\omega_i'(\sigma').
\]
Adding and scaling these relations, we then conclude that $(\Sigma,u,\lambda\omega_1+\omega_2)$ is equivalent to $(\Sigma',u',\lambda\omega_1'+\omega_2')$, as desired. The vector space axioms are straightforward to verify.
\end{proof}

We denote the associated graded vector space of Minkowski weights on $V$ by
\[
MW_\bullet(V)=\bigoplus_{d=0}^n MW_d(V).
\]

\subsection{Minkowski weights on fan sets} 

A \textbf{$d$-fan set $\cX\subseteq V$} is a subset of $V$ that can be realized as the support of a $d$-fan in $V$. A \textbf{Minkowski $k$-weight on a $d$-fan set $\cX\subseteq V$} is any Minkowski weight in $MW_k(V)$ that is supported within $\cX$. Denote the subspace of Minkowski $k$-weights on $\cX$ by
\[
MW_k(\cX)\subseteq MW_k(V)
\]
and the associated graded vector space by
\[
MW_\bullet(\cX)=\bigoplus_{k=0}^d MW_d(\cX). 
\]
The following result describes how the general Minkowski weights of this section relate to the Minkowski weights on fans studied in Subsection~\ref{subsec:minkowskiweights}.

\begin{proposition}\label{prop:fansetsweights}
If $\cX$ is a $d$-fan set and $\Sigma$ is a simplicial $d$-fan such that $|\Sigma|=\cX$, then there is a canonical inclusion
\[
MW_k(\Sigma)\subseteq MW_k(\cX)
\]
for every $k\leq d$ that is an isomorphism when $k=d$.
\end{proposition}

\begin{proof}
For each $\omega\in MW_k(\Sigma,u)\cong MW_k(\Sigma)$, we map it to the equivalence class in $MW_k(\cX)$ represented by $(\Sigma,u,\omega)$. Proposition~\ref{prop:fansets} implies that this map is injective for all $k\leq d$ and also surjective for $k=d$. We note that the map is not surjective for $k<d$ because there are Minkowski $k$-weights that are supported within $\cX$ but not supported on $|\Sigma[k]|$.
\end{proof}

\subsection{Divisors and mixed degrees}

Let $\cX$ be a $d$-fan set in $V$. Given a continuous function $\phi:\cX\rightarrow\R$, we say that $\phi$ is \textbf{piecewise linear on $\cX$} if there is some fan $\Sigma$ with $|\Sigma|=\cX$ such that $\phi\in PL(\Sigma)$. Denote the piecewise linear functions on $\cX$ by $PL(\cX)$. We observe that $PL(\cX)$ is a vector space under the usual addition and scalar multiplication of functions to $\R$---the only subtle detail to note is that $PL(\cX)$ is closed under addition: if $\phi_1$ is piecewise linear on $\Sigma_1$ and $\phi_2$ is piecewise linear on $\Sigma_2$ with $|\Sigma_1|=|\Sigma_2|=\cX$, then $\phi_1+\phi_2$ is piecewise linear on the \textbf{intersection fan}
\[
\Sigma_1\wedge\Sigma_2=\{\sigma_1\cap\sigma_2\mid\sigma_1\in\Sigma_1,\sigma_2\in\Sigma_2\}.
\] 
Denote the subspace of linear functions on $\cX$ by $L(\cX)$, and define the \textbf{vector space of divisors on $\cX$} by the quotient
\[
D(\cX)=\frac{PL(\cX)}{L(\cX)}.
\]
For any divisor $D\in D(\cX)$ and fan $\Sigma$ with $|\Sigma|\subseteq\cX$, let $D_\Sigma$ denote the natural restriction of $D$ to $|\Sigma|$. We say that $\Sigma$ is \textbf{$D$-admissible} if $D_\Sigma\in D(\Sigma)$. The next result endows $MW_\bullet(\cX)$ with the structure of a $D(\cX)$-module.

\begin{proposition}\label{prop:modulestructure}
For any $d$-fan set $\cX\subseteq V$, the vector space $MW_{\bullet}(\cX)$ is a $D(\cX)$-module under a canonical action
\[
D(\cX)\times MW_\bullet(\cX)\rightarrow MW_{\bullet-1}(\cX).
\]
More precisely, if $D\in D(\cX)$ and $\Omega\in MW_k(\cX)$, we define $D\cdot\Omega\in MW_{k-1}(\cX)$ by
\[
D\cdot\Omega=[\Sigma[k-1],u,D_\Sigma\cdot\omega].
\]
where $(\Sigma,u,\omega)$ is a representative for $\Omega$ such that $\Sigma$ is $D$-admissible.
\end{proposition}

\begin{proof}
For any $D$ and $\Omega$, we first note that it is possible to find a representative $(\Sigma,u,\omega)$ for $\Omega$ for which $\Sigma$ is $D$-admissible: let $\Sigma_1$ be a fan with $|\Sigma_1|=|\Omega|$, let $\Sigma_2$ be a fan with $|\Sigma_2|=\cX$ on which $D$ can be represented by a piecewise linear function, and let $\Sigma$ be any simplicial refinement of the intersection fan $\Sigma_1\wedge\Sigma_2$.

To prove that the action is independent of the representative $(\Sigma,u,\omega)$, suppose that $(\Sigma',u',\omega')$ is another representative for which $\Sigma'$ is $D$-admissible; we must verify that
\begin{equation}\label{eq:equivalency}
(\Sigma[k-1],u,D_\Sigma\cdot\omega)\;\;\;\text{ is equivalent to }\;\;\;(\Sigma'[k-1],u',D_{\Sigma'}\cdot\omega').
\end{equation}
Since any two fans admit a common simplicial refinement, it suffices to assume that $\Sigma'$ refines $\Sigma$. We first note that the support of each of these marked Minkowski weights is contained in $\Sigma[k-1]$---this is because, by virtue of $\Sigma$ being $D$-admissible, $D$ is linear across any cone $\tau'\in\Sigma'[k-1]$ that is not contained in some $\tau\in\Sigma(k-1)$, so $(D_{\Sigma'}\cdot\omega')(\tau')=0$. Now given any $\tau'\in\Sigma'(k-1)$ and $\tau\in\Sigma(k-1)$ such that $\tau'^\circ\cap\tau^\circ\neq\emptyset$, we must have $\tau'\subseteq\tau$, in which case it follows from \eqref{eq:bc1} and the definition of the action of $D(\Sigma)$ on $MW_k(\Sigma)$ that
\[
(D_{\Sigma'}\cdot\omega')(\tau')=\frac{\vol(\tau',u')}{\vol(\tau,u)}(D_{\Sigma}\cdot\omega)(\tau),
\]
which verifies \eqref{eq:equivalency}. The module axioms are straightforward to verify.
\end{proof}

Given a $d$-fan set $\cX\subseteq V$, a Minkowski $d$-weight $\Omega\in MW_d(\cX)$, and a collection of $d$ divisors $D_1,\dots,D_d\in D(\cX)$, we note that $D_1\cdots D_d\cdot\Omega$ is a Minkowski $0$-weight, or in other words, a function $\{0\}\rightarrow\R$. We define the \textbf{mixed degree of $D_1,\dots,D_d$ with respect to $(\cX,\Omega)$} by
\[
\deg_{\cX,\Omega}(D_1\cdots D_d)=(D_1\cdots D_d\cdot\Omega)(0)\in\R.
\]

The next result computes the general mixed degrees defined here to those defined for fixed fans in Subsection~\ref{subsec:mixeddegrees}.

\begin{proposition}\label{prop:mixeddegrees}
Let $\cX$ be a $d$-fan set, $\Omega\in MW_d(\cX)$, and $D_1,\dots,D_d\in D(\cX)$. There exists a simplicial fan $\Sigma$ supported on $\cX$ that is $D_i$-admissible for all $i=1,\dots,d$. Moreover, given such a fan $\Sigma$ and the unique Minkowski weight $\omega\in MW_d(\Sigma)$ representing $\Omega$, we have
\[
\deg_{\cX,\Omega}(D_1\cdots D_d)=\deg_{\Sigma,\omega}(D_1\cdots D_d).
\]
\end{proposition}

\begin{proof}
To find such a fan $\Sigma$, let $\Sigma_1,\dots,\Sigma_d$ be fans supported on $\cX$ such that $\Sigma_i$ is $D_i$-admissible for each $i$, and let $\Sigma$ be any simplicial common refinement of $\Sigma_1,\dots,\Sigma_d$. The equality of mixed degrees then follows from the definition of the $D(\cX)$-module structure on $MW_\bullet(\cX)$ in Proposition~\ref{prop:modulestructure} in relation to the $D(\Sigma)$-module structure on $MW_\bullet(\Sigma)$.
\end{proof}

\subsection{Convex divisors} 

Let $\cX$ be a $d$-fan set in $V$. We say that a divisor $D\in D(\cX)$ is \textbf{convex} if $D_\Sigma$ is convex (in the sense of Subsection~\ref{subsec:convex}) for some $D$-admissible fan $\Sigma$ supported on $\cX$. The following observation shows that this notion does not depend on the choice of $D$-admissible fan $\Sigma$.

\begin{proposition}\label{prop:convexeverywhere}
Let $\cX$ be a fan set in $V$ and let $D\in D(\cX)$ be a convex divisor. Then $D_\Sigma$ is convex for every $D$-admissible fan $\Sigma$.
\end{proposition}

\begin{proof}
This follows from the fact that two fans on the same support admit a common refinement, along with the fact that, for any $D$-admissible fan $\Sigma$ and refinement $\Sigma'$ of $\Sigma$, the fan $\Sigma'$ is also $D$-admissible, and $D_{\Sigma'}$ is convex on $\Sigma'$ if and only if $D_\Sigma$ is convex on $\Sigma$.
\end{proof}

We let $\overline K(\cX)\subseteq D(\cX)$ denote the set of convex divisors. This is a cone---in the sense that it is closed under addition and positive scaling---but the reader should be warned that $D(\cX)$ is generally an infinite-dimensional vector space in this setting. Because strict convexity is not well-behaved with respect to refinement, it is not a natural notion to define in the general setting of fan sets.

Given a $d$-fan set $\cX\subseteq V$, we say that a Minkowski $d$-weight $\Omega\in MW_d(\cX)$ is \textbf{nonnegative} if it has a representative $(\Sigma,\omega)$ for which $\omega\in MW_d(\Sigma)$ is nonnegative. Being nonnegative is preserved by equivalence, so knowing that one representative is nonnegative implies that all representatives are nonnegative. In the general setting of fan sets, we have the following analogue of \eqref{eq:nonnegativity}.

\begin{proposition}
Given a $d$-fan set $\cX\subseteq V$, a nonnegative Minkowski weight $\Omega\in MW_d(\cX)$, and convex divisors $D_1,\dots,D_d\in\overline K(\cX)$, we have
\[
\deg_{\cX,\Omega}(D_1,\dots,D_d)\geq 0.
\]
\end{proposition}

\begin{proof}
Using Propositions~\ref{prop:mixeddegrees} and \ref{prop:convexeverywhere}, this is a consequence of Equation \eqref{eq:nonnegativity}.
\end{proof}

We say that $\bX=(\cX,\Omega)$ is a \textbf{tropical $d$-fan variety} if $\cX$ is $d$-fan set and $\Omega\in MW_d(\cX)$ is a nonnegative Minkowski weight such that $|\Omega|=\cX$. Our ultimate aim is to study mixed degrees of convex divisors on tropical fan varieties, and especially to explore when they satisfy Alexandrov--Fenchel type inequalities. We accomplish this in the next section by introducing the notion of Lorentzian fan varieties.

\section{Lorentzian fan varieties}\label{sec:lorentzianfansets}

We now come to the definition of Lorentzian fan varieties, which builds on our earlier definition of Lorentzian fans.

\begin{definition}
A \textbf{Lorentzian $d$-fan variety} is a tropical $d$-fan variety $\bX$ that can be represented by a Lorentzian fan.
\end{definition}

We have two primary goals in this section. Our first goal is to prove that $\bX$ is Lorentzian if and only if \emph{all} of its representatives on quasiprojective fans are Lorentzian fans. In other words, the Lorentzian property on quasiprojective tropical fans is independent of the fan structure. Our second goal is to use the independence on fan structures to prove an analogue of Proposition~\ref{prop:tropicalfanAF} for Lorentzian fan varieties, showing that mixed degrees of convex divisors on Lorentzian fan varieties satisfy Alexandrov--Fenchel type inequalities, and that the sequence of mixed degrees for any pair of convex divisors is log-concave and unimodal.

\subsection{The Lorentzian property descends to fan varieties}

The purpose of this subsection is to prove the following result.

\begin{theorem}\label{thm:support}
If $(\Sigma,u,\omega)$ is any representative of a Lorentzian $d$-fan variety $\bX$ for which $\Sigma$ is quasiprojective and $|\Sigma|=\cX$, then $(\Sigma,u,\omega)$ is Lorentzian.
\end{theorem}

We will leverage the characterization in Theorem~\ref{thm:characterization} to prove Theorem~\ref{thm:support}. Before presenting the proof, we discuss several preliminary results. The first preliminary result asserts that Property~\ref{A} of Theorem~\ref{thm:characterization} depends only on the support of a fan.

\begin{lemma}\label{lem:pinchedsupport}
If $\Sigma_1$ and $\Sigma_2$ are simplicial fans with $|\Sigma_1|=|\Sigma_2|$, then $\Sigma_1$ is unpinched if and only if $\Sigma_2$ is unpinched.
\end{lemma}

\begin{proof}
The lemma is a consequence of the observation that a simplicial fan $\Sigma$ is unpinched if and only if, for every $v\in|\Sigma|$ and any linear subspace $L$ of dimension at most $d-2$, the set $(U\setminus L)\cap|\Sigma|$ remains connected for any sufficiently small neighborhood $U$ of $v$. This characterization of being unpinched only depends on the support of the fan.
\end{proof}

Stellar subdivisions will play a central role in the proof of Theorem~\ref{thm:support}; we now recall the definition. Suppose that $\Sigma$ is a fan in $V$ and choose a nonzero vector $v\in|\Sigma|$. For any $\sigma\in\Sigma$ with $v\in\sigma$ and any face $\tau\preceq\sigma$ with $v\notin\tau$, define the cone
\[
\tau_v=\Big\{au+bv\mid u\in\tau\text{ and }a,b\in\R_{\geq 0}\Big\}.
\]
The \textbf{stellar subdivision of $\Sigma$ at $v$}, denoted $\Sigma_v$, is the fan in $V$ defined by
\[
\Sigma_v=\{\sigma\mid\sigma\in\Sigma\text{ and }v\notin\sigma\}\cup\{\tau_v\mid \tau\preceq\sigma\text{ for some }\sigma\in\Sigma\text{ with }v\in\sigma\setminus\tau\}.
\]

The following result of W{\l}odarczyk \cite{Wlodarczyk} is an important tool regarding stellar subdivisions. We note that W{\l}odarczyk states this result for rational fans in vector spaces over $\Q$, but the methods readily generalize to $\R$ (see the note after Theorem~8.1 in \cite{Wlodarczyk}).

\begin{lemma}[\cite{Wlodarczyk}, Theorem A]\label{lem:factoring}
If $\Sigma$ and $\Sigma'$ are simplicial fans with the same support, then there exists a sequence of simplicial fans $\Sigma_0,\dots,\Sigma_m$ such that
\begin{enumerate}
\item $\Sigma_0=\Sigma$ and $\Sigma_m=\Sigma'$, and
\item $\Sigma_i$ is a stellar subdivision of $\Sigma_{i+1}$ or $\Sigma_{i+1}$ is a stellar subdivision of $\Sigma_i$ for every $i$.
\end{enumerate}
\end{lemma}

To prove Theorem~\ref{thm:support}, we will to show that the essence of Property \ref{B}---everything except possibly the quasiprojectivity---is preserved by stellar subdivisions. Toward that end, it is useful to have a complete understanding of the two-dimensional stars of stellar subdivisions.

\begin{lemma}\label{lem:starsofstellars}
Let $\Sigma$ be a simplicial $d$-fan in $V$ with $d\geq 2$, let $v\in\pi^\circ$ for some $\pi\in\Sigma(k)$, and let $\tau\in\Sigma_v(d-2)$. One of the following must be true:
\begin{enumerate}
\item[(i)] $v\notin|\N_\tau\Sigma_v|$, implying that $(\Sigma_v)^\tau=\Sigma^\tau$;
\item[(ii)] $v\in|\N_\tau\Sigma_v|\setminus\tau$, implying that $(\Sigma_v)^\tau=(\Sigma^\tau)_{\overline v}$ where $\overline v$ is the image of $v$ in $V^\tau$;
\item[(iii)] $\tau=\tau'_v$ for some $\tau'\in\Sigma(d-3)$ with $|\tau'(1)\cap\pi(1)|=k-1$, implying that $(\Sigma_v)^\tau=\Sigma^{\tau'\cup\pi}$ where $\tau'\cup\pi\in\Sigma(d-2)$ is the cone with rays $\tau'(1)\cup\pi(1)$;
\item[(iv)] $\tau=\tau'_v$ for some $\tau'\in\Sigma(d-3)$ with $|\tau'(1)\cap\pi(1)|=k-2$, implying that $(\Sigma_v)^\tau$ has the following structure: 
\begin{itemize}
\item the rays of $(\Sigma_v)^\tau$ are of the form $\{\rho^-,\rho^+,\rho_1,\dots,\rho_m\}$ with $V_{\rho^-}=V_{\rho^+}$, 
\item the $2$-cones are of the form $\sigma_i^{\pm}$ where $\sigma_i^\pm(1)=\{\rho_i,\rho^{\pm}\}$ for $i=1,\dots,m$.
\end{itemize}
\item[(v)] $\tau=\tau'_v$ for some $\tau'\in\Sigma(d-3)$ with $|\tau'(1)\cap\pi(1)|=k-3$, implying that $(\Sigma_v)^\tau$ is the normal fan of a triangle within a $2$-dimensional subspace of $V^\tau$.
\end{enumerate}
\end{lemma}

\begin{proof}
The classification into the five possible cases follows from the fact that the first two cases account for all situations where $v\notin\tau$, whereas the latter three cases account for all situations where $v\in\tau$. In the latter three cases, we must have $\tau=\tau'_v$ for some $\tau'\in\Sigma(d-3)$ with
\[
k-3\leq |\tau'(1)\cap\pi(1)|\leq k-1.
\]
The lower bound is because $\tau'$ and $\pi$ are both contained within a common $d$-cone and $\dim(\tau')=d-3$, and the upper bound is because $v$, and thus $\pi$, is not contained within $\tau'$.

The description of $(\Sigma_v)^\tau$ in each case follows from a careful analysis of the cones of dimension $d$ and $d-1$ in the neighborhood $N_\tau\Sigma_v$. In the rational setting, we note that stellar subdivisions correspond to blowups of toric varieties, and an algebraic geometer may recognize each of the separate cases in the lemma as a different type of torus-invariant surface inside a blowup of a toric variety.
\end{proof}

We conclude our preparation for the proof of Theorem~\ref{thm:support} with three computational results about eigenvalues of quadratic volume polynomials associated to the different types of star fans that appeared in the previous lemma.

\begin{lemma}\label{lem:2dedge}
Let $\bX$ be a tropical $2$-fan variety and let $(\Sigma,\omega)$ and $(\Sigma',\omega')$ be two representatives of $\bX$ for which $\Sigma'$ is a stellar subdivision of $\Sigma$. For any markings $u$ and $u'$, the quadratic forms $\Vol_{\Sigma,u,\omega_u}$ and $\Vol_{\Sigma',u',\omega_{u'}}$ have the same number of positive eigenvalues.
\end{lemma}

\begin{proof}
Suppose $\Sigma'=\Sigma_v$ for some $v\in|\Sigma|$. If $v\in\rho^\circ$ for some $\rho\in\Sigma(1)$, then $\Sigma=\Sigma'$, and the associated volume polynomials are equal. Suppose, then, that $v\in\sigma^\circ$ with $\sigma\in\Sigma(2)$. Denote the rays of $\sigma$ by $\rho_1$ and $\rho_2$ with markings $u_1$ and $u_2$, and let $a_1,a_2\in\R_{>0}$ be such that $a_1u_1+a_2u_2=v$. Let $\sigma_1,\sigma_2\in\Sigma'(2)$ denote the $2$-cones in $\Sigma'$ that are contained in $\sigma$, and let $\eta$ denote the ray generated by $v$. Since the conclusion is independent of the markings $u'$, we choose $u'$ to agree with $u$ on all rays of $\Sigma$ and we set $u_\eta'=v$.  A tedious computation using Proposition~\ref{prop:2dvolume} shows that
\[
\Vol_{\Sigma',u',\omega'}=\Vol_{\Sigma,u,\omega}-\frac{\omega(\sigma)}{a_1a_2}(z_\eta-a_1z_{\rho_1}-a_2z_{\rho_2})^2.
\]
It follows that $\Vol_{\Sigma',u',\omega'}$ has the same number of positive eigenvalues as $\Vol_{\Sigma,u,\omega}$ and one additional negative eigenvalue.
\end{proof}

\begin{lemma}\label{lem:spindle}
Let $\Sigma$ be simplicial $2$-fan with the structure described in Lemma~\ref{lem:starsofstellars}(iv). If $\omega\in MW_2(\Sigma,u)$, then $\Vol_{\Sigma,u,\omega}$ has exactly one positive and one negative eigenvalue.
\end{lemma}

\begin{proof}
Using notation from Lemma~\ref{lem:starsofstellars}(iv) for the rays and cones of $\Sigma$ and choosing a marking such that $u_{\rho^+}+u_{\rho^-}=0$, the balancing condition implies that $\omega(\sigma_i^+)=\omega(\sigma_i^-)$; denote this real number by $w_i$. Since the weights on $\sigma_i^\pm$ are the same, it also follows that $a_{\rho^+}=-a_{\rho^-}$; set $a=a_{\rho^-}$. Using Proposition~\ref{prop:2dvolume}, we compute
\begin{align*}
&\Vol_{\Sigma,u,\omega}=az_{\rho^+}^2-az_{\rho^-}^2+2\sum_{i=1}^mw_iz_{\rho_i}z_{\rho^+}+2\sum_{i=1}^mw_iz_{\rho_i}z_{\rho^-}\\
&=\frac{1}{4}\Big(z_{\rho^+}+z_{\rho^-}+az_{\rho^+}-az_{\rho^-}+2\sum_{i=1}^mw_iz_{\rho_i}\Big)^2-\frac{1}{4}\Big(z_{\rho^+}+z_{\rho^-}-az_{\rho^+}+az_{\rho^-}-2\sum_{i=1}^mw_iz_{\rho_i}\Big)^2,
\end{align*}
implying that $\Vol_{\Sigma,u,\omega}$ has exactly one positive and one negative eigenvalue.
\end{proof}

\begin{lemma}\label{lem:complete}
Let $\Sigma$ be the normal fan of a triangle. If $\omega\in MW_2(\Sigma,u)$ is positive, then $\Hess(\Vol_{\Sigma,u,\omega})$ has one positive eigenvalue and no negative eigenvalues.
\end{lemma}

\begin{proof}
Denote the rays of $\Sigma$ by $\rho_1,\rho_2,\rho_3$ and denote the $2$-cone with rays $\rho_i,\rho_j$ by $\sigma_{ij}$. Since the conclusion is independent of our choice of $u$, we may choose a marking such that $u_{\rho_1}+u_{\rho_2}+u_{\rho_3}=0$. It then follows from the balancing condition that there is a positive number $w$ such that $\omega(\sigma_{ij})=-a_{\rho_k}=w$ for all $i,j,k$. Proposition~\ref{prop:2dvolume} then implies that
\[
\Vol_{\Sigma,u,\omega}=w(z_{\rho_1}+z_{\rho_2}+z_{\rho_3})^2,
\]
and the result follows.
\end{proof}

We now prove that the Lorentzian property descends to tropical fan varieties.

\begin{proof}[Proof of Theorem~\ref{thm:support}]

Assume that $\bX$ is a Lorentzian $d$-fan variety and let $(\Sigma,u,\omega)$ be a representative of $\bX$ for which $\Sigma$ is quasiprojective with $|\Sigma|=\cX$. We aim to show that $(\Sigma,u,\omega)$ is Lorentzian, and by Theorem~\ref{thm:characterization}, it suffices to prove that
\begin{enumerate}
\item $\Sigma$ is unpinched, and
\item $\Vol_{\Sigma^\tau,\omega^\tau,u^\tau}$ has exactly one positive eigenvalue for every $\tau\in\Sigma(d-2)$.
\end{enumerate}

By definition of Lorentzian fan varieties, we know that $\bX$ admits a Lorentzian representative $(\Sigma',u',\omega')$ with $|\Sigma'|=\cX=|\Sigma|$. By Theorem~\ref{thm:characterization}, we know that $\Sigma'$ is unpinched, and it then follows from Lemma~\ref{lem:pinchedsupport} that $\Sigma$ is unpinched, proving (1).

To verify (2), note that Lemma~\ref{lem:factoring} allows us to find a sequence of stellar subdivisions interpolating between $\Sigma$ and $\Sigma'$, and since we know that (2) holds for $(\Sigma',u',\omega')$, this reduces the proof of (2) to studying two special cases:
\begin{enumerate}
\item[(2a)] $\Sigma$ is a stellar subdivision of $\Sigma'$, or
\item[(2b)] $\Sigma'$ is a stellar subdivision of $\Sigma$.
\end{enumerate}

Assume that we are in the setting of (2a) and let $\tau\in\Sigma(d-2)$. By Lemma~\ref{lem:starsofstellars}, there are five possibilities for $\Sigma^\tau$. In each of the cases (i) and (iii), we see that $\Sigma^\tau$ is equal to $(\Sigma')^\pi$ for an appropriate cone $\pi\in\Sigma'(d-2)$, so (2) follows from our assumptions on $(\Sigma',u',\omega')$. In case (ii), $\Sigma^\tau$ is a stellar subdivision of a two-dimensional star fan of $\Sigma'$, and the conclusion follows from Lemma~\ref{lem:2dedge} and our assumption on $(\Sigma',u',\omega')$. In cases (iv) and (v), the conclusion follows from Lemma~\ref{lem:spindle} and Lemma~\ref{lem:complete}, respectively.

Now assume that we are in the setting of (2b). Let $v\in|\Sigma|$ be the point at which we are performing the stellar subdivision, and let $\tau\in\Sigma(d-2)$. There are three cases to consider: (i) $v\notin|\N_\tau\Sigma|$, (ii) $v\in|\N_\tau\Sigma|\setminus\tau$, and (iii) $v\in\tau$---these three cases essentially correspond to the first three cases in Lemma~\ref{lem:starsofstellars}, the other two cases being irrelevant in this direction of the argument. In cases (i) and (iii), the star $\Sigma^\tau$ is equal to $(\Sigma')^\pi$ for an appropriate cone $\pi\in\Sigma'(d-2)$, and the conclusion follows from our assumption on $(\Sigma',u',\omega')$. In case (ii), there is a cone $\pi\in\Sigma'(d-2)$ for which $(\Sigma')^\pi$ is a stellar subdivision of $\Sigma^\tau$, so the conclusion follows from Lemma~\ref{lem:2dedge} and our assumption on $(\Sigma',u',\omega')$.
\end{proof}

\subsection{Alexandrov--Fenchel inequalities for Lorentzian fan varieties}

In this subsection, we prove the following analogue of Proposition~\ref{prop:tropicalfanAF} for Lorentzian fan varieties.

\begin{theorem}\label{thm:fansetAF}
Let $\bX$ be a Lorentzian $d$-fan variety. For any $D_1,\dots,D_d\in \overline K(\cX)$, we have
\[
\deg_{\bX}(D_1 D_2 D_3\cdots D_d)^2\geq \deg_{\bX}(D_1^2D_3\cdots D_d)\cdot \deg_{\bX}(D_2^2D_3\cdots D_d).
\]
Furthermore, for $D_1,D_2\in\overline K(\cX)$, the sequence
\[
\big(\deg_{\bX}(D_1^kD_2^{d-k})\big)_{k=0}^d
\]
 is log-concave and unimodal.
\end{theorem}

In order to prove Theorem~\ref{thm:fansetAF}, we would like to choose a particular representative $(\Sigma,u,\omega)$ and apply Proposition~\ref{prop:tropicalfanAF}---the only subtle point that could possibly get in the way is whether we can choose such a representative that is quasiprojective. We now present two lemmas that ensure we can always choose a quasiprojective representative.

\begin{lemma}\label{lem:final1}
If $\Sigma$ is a simplicial quasiprojective fan, then any stellar subdivision of $\Sigma$ is also quasiprojective.
\end{lemma}

\begin{proof}
In the rational setting, this is a well-known result of toric geometry \cite[Proposition~11.1.6]{CoxLittleSchenck}, and the proof generalizes to the general setting. To summarize the main point: given $v\in|\Sigma|\setminus |\Sigma[1]|$, let $D_v$ be the divisor on $\Sigma_v$ represented by the piecewise linear function that takes value $1$ at $v$ and value $0$ at $u_\rho$ for all $\rho\in\Sigma(1)$. One checks that $D-\epsilon D_v\in K(\Sigma_v)$ for any $D\in K(\Sigma)$ and sufficiently small $\epsilon>0$ \cite[Proposition~5.4]{ArdilaDenhamHuh}.
\end{proof}

\begin{lemma}\label{lem:final2}
If $\Sigma$ and $\Sigma'$ are simplicial fans in $V$ with the same support, then they admit a common refinement that can be obtained from $\Sigma$ by a sequence of stellar subdivisions.
\end{lemma}

\begin{proof}
Note that each of the finitely many cones of $\Sigma'$ is an intersection of finitely-many half-spaces, and a fan $\widetilde\Sigma$ with the same support as $\Sigma'$ is a refinement of $\Sigma'$ if each $\tilde\sigma\in\widetilde\Sigma$ lies entirely on one side of each of the hyperplanes associated to the finitely-many half-spaces defining all the cones of $\Sigma'$. Thus, it suffices to prove that, for a given hyperplane $H$, we can find a sequence of stellar subdivisions of $\Sigma$ such that every cone in the resulting fan is contained entirely on one side of $H$. To do this, we simply perform a sequence of stellar subdivisions along every ray that can be obtained as $H\cap\tau$ for some $\tau\in\Sigma(2)$; it is readily verified that the resulting fan satisfies the desired property.
\end{proof}

We are now prepared to prove Theorem~\ref{thm:fansetAF}.

\begin{proof}[Proof of Theorem~\ref{thm:fansetAF}]
By Proposition~\ref{prop:mixeddegrees}, we may choose a simplicial fan $\Sigma_1$ supported on $\cX$ that is $D_i$-admissible for each $i$. Furthermore, knowing that $\bX$ is Lorentzian tells us that there is at least one simplicial, quasiprojective fan $\Sigma_2$ supported on $\cX$. By Lemma~\ref{lem:final2}, we may find a refinement $\Sigma$ of $\Sigma_1$ that is obtained from $\Sigma_2$ by a sequence of stellar subdivisions, and by Lemma~\ref{lem:final1}, we can conclude that $\Sigma$ is simplicial and quasiprojective, and furthermore, since $\Sigma$ refines $\Sigma_1$, it follows that $\Sigma$ is $D_i$-admissible for each $i$. By Proposition~\ref{prop:convexeverywhere}, each $D_i$ is convex on $\Sigma$. Let $\omega\in MW_d(\Sigma)$ be the unique Minkowski weight such that $(\Sigma,\omega)$ represents $\bX$. By Theorem~\ref{thm:support}, we conclude that $(\Sigma,\omega)$ is Lorentzian, and the result is now a direct consequence of Propositions~\ref{prop:mixeddegrees} and \ref{prop:tropicalfanAF}.
\end{proof}

\subsection{Alexandrov--Fenchel inequalities for polytopes}

We have seemingly come a long way from the original inspiration for the ideas in this paper---the Alexandrov--Fenchel inqualities for mixed volumes of polytopes---and we now bring the discussion full-circle by briefly outlining how the Alexandrov--Fenchel inequalites sit within this story.

Consider $V=\R^n$ with basis $u_1,\dots,u_n$, and set $u_0=-\sum_{i=1}^n e_i$. Let $(\Sigma,u)$ be the marked complete fan comprised of all cones spanned by proper subsets of $\{u_0,\dots,u_n\}$. Consider the constant function $\omega:\Sigma(n)\rightarrow\{1\}$; it is straightforward to verify that $\omega\in MW_n(\Sigma,u)$ and, moreover, that every Minkowski weight is a scalar multiple of $\omega$. In addition, one readily checks that $(\Sigma,u,\omega)$ is Lorentzian; for example, one can use Theorem~\ref{thm:characterization}, though it can also be checked directly from the definition. Let $\Omega\in MW_n(\R^n)$ be the Minkowski weight represented by $(\Sigma,u,\omega)$. We have argued that $(\R^n,\Omega)$ is a Lorentzian fan variety.

Given any polytope $P\subseteq\R^n$, let $\Sigma_P\subseteq\R^n$ be its (outward) normal fan (with respect to the standard inner product on $\R^n$), and let $u$ be the unit vectors on each ray of $\Sigma_P$. We can write
\[
P=\{v\in\R^n\mid v\cdot u_{\rho}\leq a_\rho\}
\]
for some real numbers $a_\rho$, and we define a corresponding divisor
\[
D_P=\sum_{\rho\in\Sigma(1)}a_\rho D_{u,\rho}\in D(\Sigma_P)\subseteq D(\R^n).
\]
This divisor is convex, more or less by definition of the normal fan, and it turns out that
\begin{equation}\label{eq:volume}
\vol_n(P)=n!\deg_{\R^n,\Omega}((D_P)^n).
\end{equation}
Formula \eqref{eq:volume} is essentially a consequence of an inductive argument that uses \eqref{eq:reduction} to reduce the right-hand side and uses the following standard fact---which, geometrically corresponds to subdividing $P$ into pyramids over its facets $\{F_\rho\mid\rho\in\Sigma(1)\}$---to reduce the left-hand side:
\[
\vol_n(P)=\frac{1}{n}\sum_{\rho\in\Sigma(1)}a_\rho\vol_{n-1}(F_\rho).
\]
From \eqref{eq:volume}, it then follows from multilinearity and symmetry of mixed volumes that, for any polytopes $P_1,\dots,P_n\in\R^n$, the mixed volume $\vol_{n}(P_1\cdots P_n)$ can be computed by
\[
\vol_n(P_1\cdots P_n)=n!\deg_{\R^n,\Omega}(D_{P_1}\cdots D_{P_n}).
\]
Thus, we see that the Alexandrov--Fenchel inqualities for mixed volumes of polytopes
\[
\vol_n(P_1P_2P_3\cdots P_n)^2\geq\vol_n(P_1^2P_3\cdots P_n)\vol_n(P_2^2P_3\cdots P_n)
\]
 are captured by Theorem~\ref{thm:fansetAF} simply by considering the special case of $\bX=(\R^n,\Omega)$. Two interesting aspects of this proof are that (i) it does not require a proof of the two-dimensional Brunn--Minkowski inequalities, and (ii) it does not require one to approximate sets of polytopes with sets of simple strongly isomorphic polytopes.

\bibliographystyle{alpha}
\newcommand{\etalchar}[1]{$^{#1}$}

\end{document}